\documentclass{amsart}

\usepackage{amsmath}
\usepackage{amssymb}
\usepackage{commath}
\usepackage{mathtools}
\usepackage[all]{xy}

\usepackage{hyperref}

\theoremstyle{plain}
\newtheorem{theorem}{Theorem}
\newtheorem{lemma}[theorem]{Lemma}
\newtheorem{proposition}[theorem]{Proposition}
\newtheorem{corollary}[theorem]{Corollary}
\theoremstyle{definition}
\newtheorem{definition}[theorem]{Definition}
\newtheorem{remark}[theorem]{Remark}

\numberwithin{theorem}{section}
\numberwithin{equation}{section}

\newcommand{\DIV}[1]{{\mathrm{D}_{#1}^{\mathrm{IV}}}}

\newcommand{\C}{\mathbb{C}}
\newcommand{\N}{\mathbb{N}}
\newcommand{\R}{\mathbb{R}}
\newcommand{\T}{\mathbb{T}}

\newcommand{\GL}{\mathrm{GL}}

\newcommand{\SO}{\mathrm{SO}}
\newcommand{\Spe}{\mathrm{S}}
\newcommand{\Ad}{\mathrm{Ad}}
\newcommand{\U}{\mathrm{U}}
\newcommand{\mrT}[1]{\mathrm{T}_{#1}}

\newcommand{\fg}{\mathfrak{g}}
\newcommand{\fh}{\mathfrak{h}}

\newcommand{\ft}{\mathfrak{t}}

\newcommand{\so}{\mathfrak{so}}
\newcommand{\fu}{\mathfrak{u}}
\newcommand{\su}{\mathfrak{su}}
\newcommand{\gl}{\mathfrak{gl}}

\newcommand{\cT}{\mathcal{T}}
\newcommand{\cA}{\mathcal{A}}
\newcommand{\cB}{\mathcal{B}}
\newcommand{\cP}{\mathcal{P}}
\newcommand{\cS}{\mathcal{S}}
\newcommand{\HH}{\mathcal{H}}

\newcommand{\End}{\mathrm{End}}

\begin{document}

\title[Commuting Toeplitz, Cartan type IV and moment maps]{Commuting Toeplitz operators on Cartan domains of type IV and moment maps}

\author{Ra\'ul Quiroga-Barranco}
\address{Centro de Investigaci\'on en Matem\'aticas, Guanajuato, M\'exico}
\email{quiroga@cimat.mx}

\author{Monyrattanak Seng}
\address{Centro de Investigaci\'on en Matem\'aticas, Guanajuato, M\'exico}
\email{monyrattanak.seng@cimat.mx}



\maketitle

\begin{abstract}
	Let us consider, for $n \geq 3$, the Cartan domain $\mathrm{D}_n^{\mathrm{IV}}$ of type IV. On the weighted Bergman spaces $\mathcal{A}^2_\lambda(\mathrm{D}_n^{\mathrm{IV}})$ we study the problem of the existence of commutative $C^*$-algebras generated by Toeplitz operators with special symbols. We focus on the subgroup $\mathrm{SO}(n) \times \mathrm{SO}(2)$ of biholomorphisms of $\mathrm{D}_n^{\mathrm{IV}}$ that fix the origin. The $\mathrm{SO}(n) \times \mathrm{SO}(2)$-invariant symbols yield Toeplitz operators that generate commutative $C^*$-algebras, but commutativity is lost when we consider symbols invariant under a maximal torus or under $\mathrm{SO}(2)$. We compute the moment map $\mu^{\mathrm{SO}(2)}$ for the $\mathrm{SO}(2)$-action on $\mathrm{D}_n^{\mathrm{IV}}$ considered as a symplectic manifold for the Bergman metric. We prove that the space of symbols of the form $a = f \circ \mu^{\mathrm{SO}(2)}$, denoted by $L^\infty(\mathrm{D}_n^{\mathrm{IV}})^{\mu^{\mathrm{SO}(2)}}$, yield Toeplitz operators that generate commutative $C^*$-algebras. Spectral integral formulas for these Toeplitz operators are also obtained.
\end{abstract}

%\tableofcontents

\section{Introduction}
For a circled bounded symmetric domain $D$ in $\C^n$ we define the (weightless) Bergman space $\cA^2(D)$ as the space of holomorphic square-integrable functions with respect to the (normalized) Lebesgue measure $\dif v(z)$. This turns out to be a closed subspace of $L^2(D)$. Furthermore, $\cA^2(D)$ is a reproducing kernel Hilbert space and the orthogonal projection $B : L^2(D) \rightarrow \cA^2(D)$ is given by integration against the reproducing kernel. This naturally leads to the so-called Toeplitz operators, which are bounded operators on $\cA^2(D)$ of the form $T_a = B \circ M_a$, where $a \in L^\infty(D)$ is called the symbol of $T_a$. 

Bergman spaces and Toeplitz operators yield concrete examples of a Hilbert space and a family of operators to study. Furthermore, this setup is quite rich and not very restrictive. For example, it has been proved in \cite{EnglisDensity} that the Toeplitz operators are dense in the strong operator topology in the algebra of all bounded operators. This certainly implies that generically two given Toeplitz operators will not commute. Nevertheless, it has been found that by selecting suitable families of symbols the Toeplitz operators will in fact commute. More precisely, if we denote by $\cT(\cS)$ the unital Banach algebra generated by Toeplitz operators, then for many choices of $\cS \subset L^\infty(D)$ the algebra $\cT(\cS)$ is in fact a commutative $C^*$-algebra. Most of the examples of commutative $C^*$-algebras obtained in this way are infinite dimensional with very complicated structure. We refer to \cite{DOQJFA,DOQMatrix22,DQGesturVolume,GKVRadial,KorenblumZhu1995,QVUnitBall1,QVUnitBall2} for a few examples of these facts. Furthermore, even just the problem of finding commuting Toeplitz operators or studying their commutator has received some attention. We refer to \cite{AppuhamyLe2016,AxlerCuckovicRao2000,ChoeKooLee2004,Le2017,ZorbRadial2002} for some works related to these problems.

For a circled bounded symmetric domain $D$ we have at our disposal its group $\mathrm{Aut}(D)$ of biholomorphisms. For a few years now, it has been known that for  $H \subset \mathrm{Aut}(D)$ a suitably chosen subgroup, the space of $H$-invariant symbols, denoted by $L^\infty(D)^H$, yields a commutative $C^*$-algebra $\cT(L^\infty(D)^H)$. This phenomenon was first discovered in \cite{KorenblumZhu1995} for the unit disk $\mathbb{D}$ and the group $\T$ of rotations around the origin. Then, some more recent general results have been obtained for the unit ball and more general bounded symmetric domains (see the references mentioned above). In most cases, this requires a quite large subgroup $H$ to obtain commutative $C^*$-algebras generated by Toeplitz operators.

Most recently a new technique has been introduced to construct special families of symbols. The underlying symplectic geometry of the unit ball $\mathbb{B}^n$ was used in \cite{QSJFAUnitBall} to build commutative $C^*$-algebras generated by Toeplitz operators. For the case of the unit ball $\mathbb{B}^n$ it is known that maximal Abelian subgroups of $\mathrm{Aut}(\mathbb{B}^n)$ yield, by the invariance procedure described above, commutative $C^*$-algebras generated by Toeplitz operators (see \cite{QVUnitBall1,QVUnitBall2}). The results from \cite{QSJFAUnitBall} can be seen as providing an alternative symplectic-geometric interpretation of these results for the unit ball $\mathbb{B}^n$. More precisely, it was proved that, for a maximal Abelian subgroup $H \subset \mathrm{Aut}(\mathbb{B}^n)$ a symbol $a \in L^\infty(\mathbb{B}^n)$ is $H$-invariant if and only if for some function $f$ we can write $a = f\circ \mu^H$, where $\mu^H$ is the moment map for the symplectic $H$-action on $\mathbb{B}^n$. Hence, one can argue that the existence of commutative $C^*$-algebras generated by Toeplitz operators associated to a maximal Abelian subgroup $H$ is a consequence of the underlying symplectic geometry.

On the other hand, for bounded symmetric domains $D$ with rank at least $2$ it has been found that the maximal Abelian subgroups $H \subset \mathrm{Aut}(D)$ very rarely yield commutative $C^*$-algebras generated by Toeplitz operators, by simply using the symbols $L^\infty(D)^H$. This points to the unit ball $\mathbb{B}^n$, whose rank is $1$, as a very special case. Given the symplectic-geometric explanation of the situation for the unit ball, it is natural to consider, for bounded symmetric domains with higher rank, the problem of constructing commutative $C^*$-algebras generated by Toeplitz operators using moment maps. This is the setup for the main contribution of this work.

We consider the $n$-dimensional Cartan domain $\DIV{n}$ of type IV. This domain is defined in Section~\ref{sec:DIV} and we assume in the rest of this work that $n \geq 3$ to avoid the trivial lower dimensional cases. As described above, we have Bergman spaces and Toeplitz operators on the domain $\DIV{n}$, which we consider for the general weighted case. The subgroup of biholomorphisms that fix the origin is given by a linear action of $\SO(n)\times \SO(2)$. This group yields two very important actions. First, we have its symplectic action on $\DIV{n}$, and second, we have unitary representations $\pi_\lambda$ on the weighted Bergman spaces ($\lambda > n-1$) (see subsection~\ref{subsec:Bergmanmetric} and Section~\ref{sec:isotypic}, respectively). The maximal tori of $\SO(n)\times \SO(2)$ have dimension $\lfloor n/2\rfloor +1$ and we choose in subsection~\ref{subsec:isotypicSOnSO2} a (canonical) maximal torus denoted by $\mrT{n} \times \SO(2)$.

We state in Theorem~\ref{thm:comm-invariance} that the $C^*$-algebras $\cT^{(\lambda)}(L^\infty(\DIV{n})^{\SO(n) \times \SO(2)})$, corresponding to $\SO(n)\times \SO(2)$-invariant symbols, are commutative, while the $C^*$-algebras $\cT^{(\lambda)}(L^\infty(\DIV{n})^{\mrT{n} \times \SO(2)})$, corresponding to $\mrT{n} \times \SO(2)$-invariant symbols, are not commutative. The former claim was already known from \cite{DOQJFA} but the latter is new. Note that $\mrT{n} \times \SO(2)$ is a maximal Abelian subgroup of the biholomorphism group of $\DIV{n}$, thus providing an example of the behavior noted above for higher rank bounded symmetric domains. Recall that the rank of $\DIV{n}$ is $2$ for every $n$. Theorem~\ref{thm:comm-invariance} also proves that the $C^*$-algebra $\cT^{(\lambda)}(L^\infty(\DIV{n})^{\SO(n-1) \times \SO(2)})$ is not commutative. As noted in Remark~\ref{rmk:noncommMaxConn}, the subgroup $\SO(n-1) \times \SO(2)$ is maximal connected in $\SO(n) \times \SO(2)$. Hence, not even dropping the Abelian property of a subgroup of $\SO(n)\times \SO(2)$ allows us to obtain commutative $C^*$-algebras generated by Toeplitz operators by using invariant symbols for a still quite large subgroup.

In Section~\ref{sec:momentmaps} we compute in Theorem~\ref{thm:momentmapTnSO(2)} the moment map for the symplectic $\mrT{n} \times \SO(2)$-action on $\DIV{n}$. As a consequence we obtain in Corollary~\ref{cor:momentmapSO(2)} the moment map for the $\SO(2)$-action on $\DIV{n}$. In this same section, we consider moment map symbols for $\SO(2)$ as symbols of the form $a = f \circ \mu^{\SO(2)}$, where $f$ is some function and $\mu^{\SO(2)}$ is the moment map computed in Corollary~\ref{cor:momentmapSO(2)}. See also Definition~\ref{def:momentmapsymbol} for a more general notion. The space of moment map symbols for $\SO(2)$ is denoted by $L^\infty(\DIV{n})^{\mu^{\SO(2)}}$. As one of our main results, we prove in Theorem~\ref{thm:momentmapSO(2)Toeplitz} that the $C^*$-algebra $\cT^{(\lambda)}(L^\infty(\DIV{n})^{\mu^{\SO(2)}})$ is commutative. Furthermore, we also find an explicit Hilbert direct sum decomposition of $\cA^2_\lambda(\DIV{n})$ with respect to which all the elements of $\cT^{(\lambda)}(L^\infty(\DIV{n})^{\mu^{\SO(2)}})$ are simultaneously diagonalized.

The next problem is to explicitly compute the coefficients that diagonalize a Toeplitz operator $T^{(\lambda)}_a$, in terms of the symbol $a$, when the symbol belongs to $L^\infty(\DIV{n})^{\mu^{\SO(2)}}$. Such explicit formula is obtained in Theorem~\ref{thm:specintToeplitzSO(2)}. Similar formulas were obtained in \cite{DQGesturVolume}, and in fact we use those to obtain ours. The formulas from \cite{DQGesturVolume} require computing square roots for positive elements in the Jordan algebra associated to $\DIV{n}$. We achieve this in Corollary~\ref{cor:JordanRoots}, which yields a fairly complicated formula for the square roots. However, the moment map $\mu^{\SO(2)}$ and the formula for the square roots play together very well to provide the fairly simple spectral integral formula shown in Theorem~\ref{thm:specintToeplitzSO(2)}. This formula and the fact that the very small subgroup $\SO(2)$ yields commutative $C^*$-algebras using its moment map, highlight the importance of symplectic geometry in the solution of this operator theory problems. One can safely conjecture that the existence of many of the known commutative $C^*$-algebras generated by Toeplitz operators can be explained by symplectic geometric means.

As for the details of our work, one of our main tools is the use of representation theory of compact groups. In particular, we compute in Section~\ref{sec:isotypic} the isotypic decompositions for the representation $\pi_\lambda$ of the group $\SO(n) \times \SO(2)$ on the weighted Bergman spaces $\cA^2_\lambda(\DIV{n})$. We also compute such decompositions for the restrictions of $\pi_\lambda$ to $\SO(n-1) \times \SO(2)$, $\mrT{n} \times \SO(2)$ and $\SO(2)$. Then, the criterion provided by Theorem~\ref{thm:multfree-comm} allows us to determine the commutativity of the $C^*$-algebras generated by Toeplitz operators with $H$-invariant symbols in terms of the isotypic decomposition of $\cA^2_\lambda(\DIV{n})$ corresponding to $H$.

In Section~\ref{sec:momentmaps} we compute the moment maps of $\mrT{n} \times \SO(2)$ and $\SO(2)$. And in Section~\ref{sec:Toeplitzmomentmap} we prove the main results which, as noted above, involve some computations with Jordan algebras. It is worthwhile to note that the commutativity of the $C^*$-algebras $\cT^{(\lambda)}(L^\infty(\DIV{n})^{\mu^{\SO(2)}})$ can be explained by a very peculiar relationship between the groups $\SO(n) \times \SO(2)$ and $\SO(2)$. To be more precise, the moment map $\mu^{\SO(2)}$ is not only $\SO(2)$-invariant, as it must be by definition, but also $\SO(n) \times \SO(2)$-invariant. We use this to obtain the commutativity of $\cT^{(\lambda)}(L^\infty(\DIV{n})^{\mu^{\SO(2)}})$ (see the proof of Theorem~\ref{thm:momentmapSO(2)Toeplitz}) and to be able to compute the spectral integral formulas in Theorem~\ref{thm:specintToeplitzSO(2)}.

\section{The Cartan domains of type IV}\label{sec:DIV}
We will denote by $\DIV{n}$ the $n$-dimensional Cartan domain of type IV, which is given by
\[
    \DIV{n}= \{z \in \C^n \mid |z|<1,\; 2|z|^2 < 1+|z^\top z|^2\},
\]
where the elements of $\C^n$ are considered from now as columns. It is very well known that $\DIV{n}$ is a circled bounded symmetric domain for every $n \geq 1$ (see \cite{Mok1989,Upmeier1996}). For $n=1$, it is easy to see that this domain is the unit disk $\mathbb{D}$, since the second defining condition is always satisfied in this case. Also, using the arguments found in \cite{Mok1989}, it is easy to see that there is a biholomorphism between $\DIV{2}$ and $\mathbb{D}^2$. On the other hand, for every $n \geq 3$, the bounded symmetric domain $\DIV{n}$ is irreducible (see \cite{HelgasonDGLSS,Mok1989}). For this reason, we will consider the case $n \geq 3$ in the rest of this work. On the other hand, it is well known that $\so(4,2) \simeq \su(2,2)$ (see \cite{HelgasonDGLSS}) and so the domain $\DIV{4}$ is biholomorphically equivalent to the Cartan domain $\mathrm{D}^{\mathrm{I}}_{2\times2}$. The latter is defined as the set of complex $2 \times 2$ matrices $Z$ that satisfy $Z^*Z < I_2$.

The connected component of the biholomorphism group of $\DIV{n}$ is isomorphic to $\SO_0(n,2)$. The proof of this can be found in \cite{HelgasonDGLSS,Mok1989,Satake1980}. In particular, \cite{Satake1980} shows that the connected component of the biholomorphism group of $\DIV{n}$ is realized through an action of $\SO_0(n,2)$ by quadratic fractional transformations. On the other hand, the maximal compact subgroup $\SO(n)\times\SO(2)$ of $\SO_0(n,2)$ acts linearly on $\DIV{n}$ by the expression
\begin{align}\label{eq:SOnSO2Action}
  \SO(n)\times\SO(2) \times \DIV{n} &\rightarrow \DIV{n} \\
  (A,t)\cdot z &= tAz, \notag
\end{align}
and this action realizes the biholomorphisms of $\DIV{n}$ that fix the origin. Note that we have used the natural isomorphism $\SO(2) \simeq \T$, where $\T$ denotes the unit circle in $\C$, that comes from the identification $\C \simeq \R^2$. More precisely, we consider the isomorphism
\begin{align}
	\T &\rightarrow \SO(2) \label{eq:TSO(2)} \\
	e^{i\vartheta} &\mapsto 
		\begin{pmatrix}
			\cos(\vartheta) & \sin(\vartheta) \\
			-\sin(\vartheta) & \cos(\vartheta)
		\end{pmatrix}, \notag
\end{align}
where $\vartheta \in \R$. We will use such identifications in the rest of this work.

From the previous remarks, it follows that we have a natural diffeomorphism
\[
    \DIV{n} \simeq \SO_0(n,2)/\SO(n)\times\SO(2)
\]
that is in fact a biholomorphism for the structure of Hermitian symmetric space for the quotient on the right. We refer to \cite{HelgasonDGLSS,Mok1989,Satake1980} for further details. We observe that $\SO(n)$ is a simple Lie group except for $n=2,4$. Nevertheless, we recall our standing assumption $n \geq 3$. 

In the notation of \cite{Upmeier1996}, the domain $\DIV{n}$ has genus $n$, rank $2$ and characteristic multiplicities $a = n-2$ and $b = 0$. In particular, the vanishing of the last value implies that $\DIV{n}$ is a bounded domain with a tube-type unbounded realization.

\subsection{Bergman spaces and Toeplitz operators}
Let us consider the Lebesgue measure $\dif v(z)$ on $\C^n$ normalized so that $v(\DIV{n}) = 1$. Then, we will denote by $\cA^2(\DIV{n})$ the (weightless) Bergman space over $\DIV{n}$ which consists of the holomorphic functions that belong to $L^2(\DIV{n},v)$. It is well known (see \cite{HelgasonDGLSS,Upmeier1996}) that this is a closed subspace of $L^2(\DIV{n},v)$ that admits a reproducing kernel. In particular, the orthogonal projection $B : L^2(\DIV{n},v) \rightarrow \cA^2(\DIV{n})$, called the (weightless) Bergman projection, is given by
\[
    B(f)(z) = \int_{\DIV{n}} f(w) K(z,w) \dif v(w),
\]
for every $f \in L^2(\DIV{n},v)$ and $z \in \DIV{n}$, where $K$ is the (weightless) Bergman reproducing kernel. The results from \cite{HuaBook} (see also \cite{LoosBSDJ}) imply that the function $K : \DIV{n}\times\DIV{n} \rightarrow \C$ is given by
\[
    K(z,w) = \big(1 -2z^\top \overline{w} + (z^\top z)\overline{(w^\top w)}\big)^{-n}.
\]

Following Section 2.9 from \cite{Upmeier1996}, we consider the weighted measure
\[
    \big(1 -2|z|^2 + |z^\top z|^2)^{\lambda - n} \dif v(z)
\]
which is finite on $\DIV{n}$ precisely for $\lambda > n-1$. Hence, for every such $\lambda$, we choose a normalizing constant $c_\lambda > 0$ such that the measure
\[
    \dif v_\lambda(z) = c_\lambda \big(1 -2|z|^2 + |z^\top z|^2)^{\lambda - n} \dif v(z)
\]
satisfies $v_\lambda(\DIV{n}) = 1$. From this, we define for every $\lambda > n-1$, the weighted Bergman space $\cA^2_\lambda(\DIV{n})$ as the subspace of holomorphic functions that belong to $L^2(\DIV{n},v_\lambda)$. As before, this is a closed subspace of $L^2(\DIV{n},v_\lambda)$ that admits a reproducing kernel. So that the orthogonal projection, called the Bergman projection, $B_\lambda : L^2(\DIV{n},v_\lambda) \rightarrow \cA^2_\lambda(\DIV{n})$ is given by
\[
    B_\lambda(f)(z) = \int_{\DIV{n}} f(w) K_\lambda(z,w) \dif v_\lambda(w),
\]
for every $f \in L^2(\DIV{n},v)$ and $z \in \DIV{n}$, where the weighted Bergman reproducing kernel $K_\lambda : \DIV{n}\times\DIV{n} \rightarrow \C$ has the expression (see \cite{Upmeier1996})
\[
    K_\lambda(z,w) = \big(1 -2z^\top \overline{w} + (z^\top z)\overline{(w^\top w)}\big)^{-\lambda}.
\]
In particular, the case $\lambda = n$ corresponds to the above weightless case. In other words, $v = v_n$, $\cA^2(\DIV{n}) = \cA^2_n(\DIV{n})$ and $K = K_n$.

We note that all the measures $v_\lambda$ for $\lambda > n-1$ yield the same null subsets. Hence, we can consider $L^\infty(\DIV{n})$ as given for any such measure. For every $a \in L^\infty(\DIV{n})$ we define the Toeplitz operator with symbol $a$ acting on $\cA^2_\lambda(\DIV{n})$ (for $\lambda > n-1$) as the bounded operator $T_a = T^{(\lambda)}_a \in \cB(\cA^2_\lambda(\DIV{n}))$ given by $T^{(\lambda)}_a = B_\lambda \circ M_a$. In particular, we have
\[
    T^{(\lambda)}_a(f)(z) = \int_{\DIV{n}} a(w) f(w) K_\lambda(z,w) \dif v_\lambda(w),
\]
for every $f \in \cA^2_\lambda(\DIV{n})$ and $z \in \DIV{n}$.

For a given subspace $\cS \subset L^\infty(\DIV{n})$, we will denote by $\cT^{(\lambda)}(\cS)$ the unital Banach algebra generated by Toeplitz operators whose symbols belong to $\cS$. In this work we will consider subspaces $\cS$ that are self-adjoint in $L^\infty(\DIV{n})$. From the elementary properties of Toeplitz operators it follows that if $\cS$ is self-adjoint, then $\cT^{(\lambda)}(\cS)$ is also the unital $C^*$-algebra generated by Toeplitz operators with symbols in $\cS$.

\subsection{The Bergman metric of $\DIV{n}$}
\label{subsec:Bergmanmetric}
We recall (see \cite{HelgasonDGLSS,Mok1989}) that for any bounded symmetric domain, there is a naturally associated Hermitian metric so that the biholomorphisms act by isometries. More precisely, for an $n$-dimensional bounded symmetric domain $D$ with (weightless) reproducing kernel $K$ this Hermitian metric, called the Bergman metric, is given by
\[
    g_z = c \sum_{j,k=1}^{n} \frac{\partial^2}{\partial z_j \partial \overline{z}_k} \log K(z,z) \dif z_j \otimes \dif \overline{z}_k,
\]
where $c > 0$ is any conveniently chosen normalizing constant. We will apply this construction to our domain $\DIV{n}$.

In the rest of this work, we will denote by $\Delta : \DIV{n} \rightarrow \C$ the function given by
\begin{equation}\label{eq:Delta}
    \Delta(z) = 1 - 2|z|^2 + |z^\top z|^2.
\end{equation}
In particular, we have $K_\lambda(z,z) = \Delta(z)^{-\lambda}$, for every $z \in \DIV{n}$ and $\lambda > n-1$. It is well known that this completely determines the weighted Bergman kernels.

The following result is a consequence of the previous remarks and straightforward computations. Nevertheless, we show some of these computations since they will be useful latter on. 

\begin{proposition}\label{prop:BergmanMetric}
    For the domain $\DIV{n}$, the Bergman metric is given by
    \begin{align*}
      g_z &= \frac{1}{2n} \sum_{j,k=1}^{n} \frac{\partial^2}{\partial z_j \partial \overline{z}_k} \log K(z,z) \dif z_j \otimes \dif \overline{z}_k \\
          &= \sum_{j,k=1}^{n} \frac{\Delta(z)\big(\delta_{jk} - 2z_j\overline{z}_k\big) + 2\big(\overline{z}_j - z_j \overline{(z^\top z)}\big) \big(z_k - \overline{z}_k (z^\top z)\big)}{\Delta(z)^2} \dif z_j \otimes \dif \overline{z}_k,
    \end{align*}
    for every $z \in \DIV{n}$.
\end{proposition}
\begin{proof}
    In the first place we have
    \begin{align*}
      \frac{\partial}{\partial z_j} \log K(z,z)
        &= -n \frac{\partial}{\partial z_j} \log \Delta(z)
            = -\frac{n}{\Delta(z)}  \frac{\partial\Delta}{\partial z_j}(z) \\
        &= \frac{2n}{\Delta(z)}(\overline{z}_j - z_j \overline{(z^\top z)}).
    \end{align*}
    And from this it follows that
    \begin{align*}
      \frac{\partial^2}{\partial z_j \partial \overline{z}_k} \log K(z,z)
        &= 2n \frac{\partial}{\partial \overline{z}_k} \Bigg(\frac{\overline{z}_j - z_j \overline{(z^\top z)}}{\Delta(z)}\Bigg) \\
        &= 2n \frac{\Delta(z)\big(\delta_{jk} - 2z_j\overline{z}_k\big) + 2\big(\overline{z}_j - z_j \overline{(z^\top z)}\big) \big(z_k - \overline{z}_k (z^\top z)\big)}{\Delta(z)^2},
    \end{align*}
    which yields the result.
\end{proof}

We recall some differential geometric notions that will be applied to the domain $\DIV{n}$. We refer to \cite{HelgasonDGLSS,Mok1989} for further details.

If $M$ is a complex manifold with a Hermitian metric $g$ and complex structure $J$, then the associated $2$-form $\omega$ is given by
\[
    \omega(X,Y) = g(JX,Y)
\]
for every pair of smooth vector fields in $M$. In particular, the $2$-form $\omega$ defines a non-degenerate anti-symmetric bilinear form at every point of $M$. The complex manifold with the Hermitian metric $g$ is called K\"ahler if the associated $2$-form $\omega$ is closed. If this is the case, $\omega$ is a non-degenerate closed $2$-form, and $M$ is called a symplectic manifold with symplectic form $\omega$.

It is well known (see \cite{HelgasonDGLSS,Mok1989}) that any bounded domain with its Bergman metric is K\"ahler. We state this in the following result for the domain $\DIV{n}$ and also give the explicit expression for the associated symplectic form.

\begin{proposition}\label{prop:SympForm}
    The domain $\DIV{n}$ with its Bergman metric $g$ as given in Proposition~\ref{prop:BergmanMetric} is a K\"ahler manifold. Furthermore, the associated symplectic form is given~by
    \[
        \omega_z = i \sum_{j,k=1}^{n} \frac{\Delta(z)\big(\delta_{jk} - 2z_j\overline{z}_k\big) + 2\big(\overline{z}_j - z_j \overline{(z^\top z)}\big) \big(z_k - \overline{z}_k (z^\top z)\big)}{\Delta(z)^2} \dif z_j \wedge \dif \overline{z}_k,
    \]
    for every $z \in \DIV{n}$, where $\Delta(z) = 1 - 2|z|^2 + |z^\top z|^2$.
\end{proposition}
\begin{proof}
    The first claim was already noted above. With our definition of associated symplectic form, the formula in the statement follows from the well known fact (see~\cite{Mok1989}) that for any K\"ahler manifold with metric given by
    \[
        g = \sum_{j,k=1}^{n} g_{jk} \dif z_j \otimes \dif \overline{z}_k,
    \]
    the associated symplectic form is given by
    \[
        \omega = i\sum_{j,k=1}^{n} g_{jk} \dif z_j \wedge \dif \overline{z}_k.
    \]
\end{proof}

In the rest of this work we will consider $\DIV{n}$ both as a K\"ahler and a symplectic manifold with the K\"ahler metric $g$ and the symplectic form $\omega$ given in Propositions~\ref{prop:BergmanMetric} and \ref{prop:SympForm}, respectively. We will also use the well known fact that every biholomorphism of $\DIV{n}$ preserves both $g$ and $\omega$ (see \cite{HelgasonDGLSS}). In particular, every biholomorphism of $\DIV{n}$ is a symplectomorphism of $(\DIV{n}, \omega)$.

\section{The representation of $\SO(n)\times\SO(2)$ and commutativity of $C^*$-algebras}\label{sec:isotypic}
We will describe in this section, the induced action of $\SO(n)\times\SO(2)$ on the Bergman spaces from the action given in \eqref{eq:SOnSO2Action}. First, we will compute natural decompositions for the Bergman spaces associated to some subgroups of $\SO(n) \times \SO(2)$. This will occupy the first subsections. Next, in the last subsection, we will consider Toeplitz operators with invariant symbols for such subgroups and determine in which cases those operators generate commutative $C^*$-algebras.

The first thing to note is that a straightforward computation shows that the function $\Delta$ given in \eqref{eq:Delta} is $\SO(n)\times\SO(2)$-invariant. Clearly, the Lebesgue measure is also $\SO(n)\times\SO(2)$-invariant, and so it follows from the definition of the weighted measures $v_\lambda$ that these have the same invariance under $\SO(n)\times\SO(2)$. Hence, we easily conclude the following result. Note that one uses the fact that the $\SO(n)\times\SO(2)$-action on $\DIV{n}$ is complex linear and so holomorphic.

\begin{proposition}\label{prop:pilambda}
   For every $\lambda > n-1$, the action given by
   \begin{align*}
     \pi_\lambda : \SO(n)\times\SO(2) \times \cA^2_\lambda(\DIV{n}) &\rightarrow \cA^2_\lambda(\DIV{n}) \\
     (\pi_\lambda(A,t)\cdot f) (z) &= f(\overline{t}A^{-1}z)
   \end{align*}
   yields a unitary representation that is continuous in the strong operator topology.
\end{proposition}

One should compare this representation with the holomorphic discrete series representations of the group $\SO_0(n,2)$ on the Bergman spaces $\cA^2_\lambda(\DIV{n})$. In fact, our representation $\pi_\lambda$ is basically the restriction of an holomorphic discrete series representation. The only difference is that, in our case, we have dropped a factor given by a character in $\SO(2)$ which is independent of $z$. This variation does not change the invariant and irreducible subspaces, but it does simplifies our computations. On the other hand, these remarks imply the well known continuity claimed at the end of Proposition~\ref{prop:pilambda}.

In the rest of this section we will describe some decompositions associated to the representations from Proposition~\ref{prop:pilambda}. Hence, we will recall some basic notation and refer to standard representation theoretic texts for further details. From now on we will only consider unitary representations that are continuous in the strong operator topology.

If $\HH_1$ and $\HH_2$ are two Hilbert spaces that admit unitary representations $\pi_1$ and $\pi_2$, respectively, of a Lie group $G$, then an intertwining map or a $G$-equivariant map is a bounded operator $T : \HH_1 \rightarrow \HH_2$ such that
\[
    T\circ\pi_1(g) = \pi_2(g)\circ T
\]
for every $g \in G$. We say that two such representations are unitary equivalent or isomorphic over $G$ when there exists an intertwining unitary map from $\HH_1$ onto $\HH_2$. On the other hand, in the case $\HH = \HH_1 = \HH_2$ and $\pi_1 = \pi_2$, we denote by $\End_G(\HH)$ the space of intertwining operators $\HH \rightarrow \HH$. It is easily seen that $\End_G(\HH)$ is a von Neumann algebra.

For a unitary representation $\pi$ of a compact group $G$ on a Hilbert space $\HH$, a $G$-invariant subspace is a closed subspace $V \subset \HH$ so that $\pi(g)(V) \subset V$ for every $g \in G$. In this case, $V$ is called irreducible over $G$ (or simply irreducible if the choice of $G$ is clear) when $V$ does not contain a non-trivial $G$-invariant subspace. Then, it is well known (see \cite{BtD}) that $\HH$ is the Hilbert direct sum of irreducible subspaces. Furthermore, one can choose a collection of irreducible representations $\pi_j : G \times V_j \rightarrow V_j$ (necessarily finite dimensional, by the compactness assumption on $G$), for every $j \in A$ ($A$ some index set) mutually non-isomorphic over $G$ so that the following are satisfied.
\begin{enumerate}
  \item The sum of all irreducible subspaces $V \subset \HH$ isomorphic over $G$ to $V_j$ is a closed non-zero subspace $\HH_j$.
  \item There is a Hilbert direct sum decomposition
    \[
        \HH = \bigoplus_{j\in A} \HH_j,
    \]
    which is preserved by $\pi(g)$ for every $g \in G$. This is called the isotypic decomposition of the representation of $\pi$ (or of $G$), and the subspaces are called the isotypic components corresponding to $\pi$ (or to $G$).
\end{enumerate}
With the previous notation, we say that an isotypic decomposition is multiplicity-free when every isotypic component is irreducible.

In the rest of this work, for a subgroup $H$ of $\SO(n)\times\SO(2)$ we will denote by $\pi_\lambda|_H$ the unitary representation of $H$ obtained by restricting $\pi_\lambda$ to $H$. We will now describe the isotypic decomposition of $\pi_\lambda|_H$ for some subgroups $H$ of $\SO(n)\times\SO(2)$.

\subsection{Isotypic decomposition for $\SO(2)$}
We will denote by $\cP(\C^n)$ the space of (holomorphic) polynomials in $\C^n$. In particular, we have $\cP(\C^n) \subset \cA^2_\lambda(\DIV{n})$ for every $\lambda > n-1$, since the measure $v_\lambda$ is finite in this case. We will also denote by $\cP^m(\C^n)$, for every $m \in \N$, the subspace of homogeneous polynomials of degree $m$. In particular, we have the algebraic direct sum
\[
    \cP(\C^n) = \bigoplus_{m=0}^\infty \cP^m(\C^n).
\]

On the other hand, it is well known (see \cite{Upmeier1996}) that the space $\cP(\C^n)$ is dense in $\cA^2_\lambda(\DIV{n})$. In fact, this is a property shared by all weighted Bergman spaces over circled bounded symmetric domains.

We note that for every $p \in \cP^m(\C^n)$ and $t \in \SO(2)$ we have
\[
    (\pi_\lambda(t) p)(z) = t^{-m} p(z) = \chi_{-m}(t) p(z),
\]
where $\chi_{-m}(t) = t^{-m}$ is the character associated to $-m$. Since the irreducible representations of $\SO(2)$ are $1$-dimensional and determined up to isomorphism by their character, it follows that the subspaces $\cP^m(\C^n)$, for $m \in \N$, are the isotypic components of the unitary representation $\pi_\lambda|_{\SO(2)}$. In particular, these subspaces are mutually orthogonal on every weighted Bergman space.

The previous remarks have proved the following result.

\begin{proposition}\label{prop:isoSO(2)}
    For every $\lambda > n-1$, the isotypic decomposition of $\pi_\lambda|_{\SO(2)}$ is given by
    \[
        \cA^2_\lambda(\DIV{n}) = \bigoplus_{m=0}^\infty \cP^m(\C^n).
    \]
    In particular, this isotypic decomposition is not multiplicity-free unless $n = 1$.
\end{proposition}

\subsection{Isotypic decomposition for $\SO(n)\times\SO(2)$}
\label{subsec:isotypicSOnSO2}
From the definition of $\pi_\lambda$ it is clear that, for every $A \in \SO(n)$, the operator $\pi_\lambda(A,1)$ intertwines the representation $\pi_\lambda|_{\SO(2)}$. In symbols, we can write
\[
    \pi_\lambda(\SO(n)) \subset \End_{\SO(2)}(\cA^2_\lambda(\DIV{n})).
\]
Hence, a straightforward application of Schur's Lemma implies that the unitary representation $\pi_\lambda|_{\SO(n)}$ preserves the Hilbert direct sum
\[
    \cA^2_\lambda(\DIV{n}) = \bigoplus_{m=0}^\infty \cP^m(\C^n).
\]
A more elementary proof of this claim is obtained by recalling that $\SO(n)$ acts linearly on $\DIV{n}$, and so its action on Bergman spaces preserves polynomials as well as their degree when they are homogeneous. At any rate, we conclude that the isotypic decomposition of $\pi_\lambda|_{\SO(n)\times\SO(2)}$ is given by a Hilbert direct sum that refines the one given above.

In order to describe the isotypic decomposition for $\SO(n)\times\SO(2)$ we will recall the basic theory of irreducible representations through the use of weights. We refer to \cite{HelgasonDGLSS,KnappBeyond} for further details. 

We start by describing some of the fundamentals of roots and root spaces. If $\U$ is a compact semisimple Lie group with Lie algebra $\fu$, and $\mathrm{T}$ is a maximal torus of $\U$ with Lie subalgebra $\ft_0$, then we can complexify the objects involved to obtain the so called roots and root spaces. More precisely, let $\fg = \fu^\C = \fu \oplus i\fu$ be the complexification of $\fu$, where $i$ here represents the complex structure associated to such complexification. Then, $\ft = \ft_0 \oplus i \ft_0$ is a Cartan subalgebra of $\fg$. For every $\alpha \in \ft^*$, we denote
\[
	\fg_\alpha = \{X \in \fg \mid [H,X] = \alpha(H)X, \text{ for every } H \in \ft\}.
\]
A linear functional $\alpha$ is called a root if both $\alpha$ and $\fg_\alpha$ are non-zero. In this case, the space $\fg_\alpha$ is called a root space. It is a well known fact that every root is real-valued in the subspace $\ft_\R = i\ft_0$, and so we identify the set of roots with their restrictions to $\ft_\R$. In other words, we consider the roots as elements of $\ft_\R^*$. The next step is to introduce an order, for which one usually considers the lexicographic order with respect to some basis. Once this is given, a root is called simple if it is not the sum of two positive roots.

Let us specialize the above information to the case of $\SO(n)$. The first object to consider is a maximal torus of the group $\SO(n)$. For this, we denote for $\vartheta \in \R$ the matrix
\begin{equation}\label{eq:Rvartheta}
    R(\vartheta) =
    \begin{pmatrix}
      \cos(\vartheta) & \sin(\vartheta) \\
      -\sin(\vartheta) & \cos(\vartheta)
    \end{pmatrix}.
\end{equation}
Then, we let $\mrT{n}$ be the subgroup of $\SO(n)$ consisting of the block diagonal matrices with size $n\times n$ of the form
\begin{equation}\label{eq:Atheta}
    A(\theta) =
    \mathrm{diag}(R(\theta_1), \dots, R(\theta_j), \dots)
\end{equation}
with $\theta \in \R^{\lfloor n/2 \rfloor}$, and where the diagonal blocks have size $2 \times 2$, except in the case of $n$ odd where there is a trailing $1$ in the diagonal position $(n,n)$. Here we are using the notation $\lfloor x \rfloor$ for the integer part of a real number $x$.
Then, it is well known that $\mrT{n}$ is a maximal torus of $\SO(n)$, and it is clear that $\dim \mrT{n} = \lfloor n/2 \rfloor$. The Lie algebra of $\mrT{n}$ will be denoted by $\ft_0$. It follows that $\mrT{n}\times \SO(2)$ is a maximal torus of $\SO(n)\times \SO(2)$ with Lie algebra $\ft_0 \times \so(2)$. For the matrix
\begin{equation}\label{eq:Rpi2}
	R(\pi/2) =
	\begin{pmatrix}
		0 & 1 \\
		-1 & 0
	\end{pmatrix},
\end{equation}
the Lie algebra $\ft_0$ consists of matrices with the form
\begin{equation}\label{eq:ft0}
	\mathrm{diag}(b_1R(\pi/2), \dots, b_jR(\pi/2), \dots)
\end{equation}
with $b \in \R^{\lfloor n/2 \rfloor}$, and where, similarly to the case of $\mrT{n}$, the diagonal blocks have size $2\times 2$, except in the case of $n$ odd where there is a trailing $0$ in the diagonal position~$(n,n)$.

The complexification of $\so(n)$ is the Lie algebra $\so(n,\C)$ of anti-symmetric complex $n\times n$ matrices. For these choices we obtain by complexification the Cartan subalgebra $\ft = \ft_0 \oplus i \ft_0$ of $\so(n,\C)$, whose elements are of the form
\[
	H(b) =
	\mathrm{diag}(ib_1R(\pi/2), \dots, ib_jR(\pi/2), \dots)
\]
with $b \in \C^{\lfloor n/2 \rfloor}$ and the same sort of restrictions imposed above. With this notation, we define for every $j = 1, \dots, \lfloor n/2 \rfloor$ the linear functional
\begin{align*}
  e_j : \ft &\rightarrow \C \\
  e_j(H(b)) &= b_j.
\end{align*}
Note that these yield a basis for $\ft^*$ and are real-valued on $\ft_\R$. It is well known that the set of roots of $\so(n,\C)$ for this choice of Cartan subalgebra is given in terms of this basis as follows.
\begin{itemize}
  \item If $n = 2\ell$, then the set of roots consists of the linear functionals $e_j - e_k$ ($1 \leq j\not=k \leq \ell$) and $\pm(e_j + e_k)$ ($1 \leq j < k \leq \ell$).
  \item If $n = 2\ell+1$, then the set of roots consists of the linear functionals $\pm e_j$ ($1 \leq j \leq \ell$),  $e_j - e_k$ ($1 \leq j\not=k \leq \ell$) and $\pm(e_j + e_k)$ ($1 \leq j < k \leq \ell$).
\end{itemize}
On the other hand, with respect to the lexicographic order induced by the basis given by the functionals $e_j$, the sets of positive and of simple roots are given as follows.
\begin{itemize}
  \item If $n = 2\ell$, then the set of positive roots consists of the linear functionals $e_j \pm e_k$ ($1 \leq j < k \leq \ell$). While the set of simple roots consists of the linear functionals $e_1 - e_2, \dots, e_{\ell -1} - e_\ell, e_{\ell -1} + e_\ell$.
  \item If $n = 2\ell + 1$, then the set of positive roots consists of the linear functionals $e_j \pm e_k$ ($1 \leq j < k \leq \ell$) and $e_j$ ($1 \leq j \leq \ell$). While the set of simple roots consists of the linear functionals $e_1 - e_2, \dots, e_{\ell -1} - e_\ell, e_\ell$.
\end{itemize}

Similar constructions can be applied to irreducible representations thus yielding the notion of weights. We will sketch some facts about the theory of weights for our case, and refer to the standard literature for further details (see \cite{GW,KnappBeyond}).

Let $\pi : \SO(n) \rightarrow \GL(V)$ be an irreducible representation, where $V$ is a complex vector space and $\GL(V)$ is the group of complex linear isomorphisms of $V$. We recall that the compactness of $\SO(n)$ and the irreducibility of $V$ imply the finite dimensionality of $V$, even if we assume that $V$ is Hilbert from the start.
Differentiation yields a representation $\rho = \dif\pi : \so(n) \rightarrow \gl(V)$, in other words an homomorphism of Lie algebras. By complexifying, we obtain an induced representation $\rho^\C : \so(n,\C) \rightarrow \gl(V)$. A linear functional $\lambda \in \ft^*$ is called a weight if the subspace
\[
    V_\lambda = \{ v \in V \mid \rho^\C(X)v = \lambda(X) v, \text{ for every } X \in \ft \}
\]
is non-zero, in which case $V_\lambda$ is called a weight space, both corresponding to the representation on $V$ given by either $\pi$ or $\rho^\C$. Let us denote by $\Phi_V$ the set of weights for the irreducible module $V$. It is very well known that the elements of $\Phi_V$ restricted to $\ft_\R$ are real-valued and, as with the roots, we will denote such restrictions with the same symbols. In particular, $\Phi_V$ is considered as a subset of $\ft_\R^*$ and so we can use the order of the latter. More precisely, for any two $\mu_1, \mu_2 \in \Phi_V$ we consider the order given by 
\[
	\mu_1 \prec \mu_2 \Longleftrightarrow \mu_2 - \mu_1 = \alpha_1 + \dots + \alpha_k
\]
for some $\alpha_1, \dots, \alpha_k$ positive roots. It follows from representation theory that any (finite-dimensional) irreducible representation $V$ as above admits a unique highest weight and any two such representations are isomorphic, as $\SO(n)$-modules, if and only if their highest weights are the same. Furthermore, the highest weight space is $1$-dimensional and any (non-zero) of its elements is called a highest weight vector, In particular, a highest weight vector for $V$ is well defined up to a constant.

The claims stated above come from what is known as the Highest Weight Theorem, which also provides a complete description of all elements of $\ft_\R^*$ that are highest weights for some irreducible representation (see \cite{KnappBeyond}).

We will use the previous notation and facts to describe the isotypic components for the representation $\pi_\lambda|_{\SO(n)\times\SO(2)}$ on $\cP^m(\C^n)$. The main object used to achieve this is the space of harmonic polynomials that we now introduce. For the rest of this work, it is useful to keep in mind that $z^\top z = z_1^2 + \dots + z_n^2$, for every $z \in \C^n$. This notation will simplify some of our formulas.

\begin{definition}\label{def:harmonic}
    Let us define by
    \[
        \partial(z^\top z) = \sum_{j=1}^{n} \frac{\partial^2}{\partial z_j^2}
    \]
    the complex holomorphic Laplacian on $\C^n$. A polynomial $p(z) \in \cP(\C^n)$ is called harmonic if it satisfies $\partial(z^\top z) p(z) = 0$. The space of all harmonic polynomials is denoted by $\HH(\C^n)$.
\end{definition}

In other words, $\HH(\C^n)$ is the kernel of the linear map $\partial(z^\top z) : \cP(\C^n) \rightarrow \cP(\C^n)$. For such map we clearly have $\partial(z^\top z)(\cP^m(\C^n)) \subset \cP^{m-2}(\C^n)$, and so it follows that if we denote by
\[
    \HH^m(\C^n) = \HH(\C^n) \cap \cP^m(\C^n),
\]
the space of harmonic polynomials homogeneous of degree $m$, then we have an algebraic direct sum
\[
    \HH(\C^n) = \bigoplus_{m=0}^{\infty} \HH^m(\C^n).
\]
The next result is a consequence of the theory found in \cite{GW,KnappBeyond}. We provide a sketch of the proof whose details can be completed easily from basic representation theory. 

\begin{proposition}\label{prop:harmonic-irred}
	Let $n \geq 3$, be given. Then, for every $m \in \N$, the subspace $\HH^m(\C^n)$ is an irreducible $\SO(n)$-module with highest weight $me_1$ and highest weight vector $(z_1 - i z_2)^m$.
\end{proposition}
\begin{proof}
    Theorem 5.6.11 from \cite{GW} states the required irreducibility for $\SO(n,\C)$. Hence, a standard unitary trick argument shows that $\HH^m(\C^n)$ is an irreducible $\SO(n)$-module.

    On the other hand, the (straightforward) computations used in the example in page 277 from \cite{KnappBeyond} show that $(z_1 - i z_2)^m$ is a weight vector for the $\SO(n)$-module $\cP^m(\C^n)$ with weight $me_1$. It is also straightforward to check that $(z_1 - i z_2)^m \in \HH^m(\C^n)$, and so $me_1$ and $(z_1 - i z_2)^m$ are a weight and a weight vector, respectively, for $\HH^m(\C^n)$. Finally, Problem 2 in page 339 from \cite{KnappBeyond} proves that the weight $me_1$ is the highest weight of $\HH^m(\C^n)$. We note that this same Problem 2 also shows the irreducibility of $\HH^m(\C^n)$ over $\SO(n)$.
\end{proof}

We proceed to describe the main isotypic decomposition of this subsection. We recall that if $G\times H$ is a product of two compact groups and $V$, $W$ are modules over $G$ and $H$, respectively, then $V\otimes W$ is a $G\times H$-module in a natural way (see~\cite{BtD}). This is called the outer tensor product of the corresponding modules.

\begin{proposition}\label{prop:isoSO(n)SO(2)}
    For every $\lambda > n-1$, the isytopic decomposition of $\pi_\lambda|_{\SO(n)\times\SO(2)}$ is given by
    \[
        \cA^2_\lambda(\DIV{n})
            = \bigoplus_{m=0}^{\infty}
                \bigoplus_{\substack{k_1, k_2 \in \N \\ k_1 + 2k_2 = m}} \HH^{k_1}(\C^n) (z^\top z)^{k_2}.
    \]
    Furthermore, this decomposition satisfies the following properties
    \begin{enumerate}
      \item For every $k_1, k_2 \in \N$, the subspace $\HH^{k_1}(\C^n) (z^\top z)^{k_2}$ is an irreducible $\SO(n)$-module isomorphic to $\HH^{k_1}(\C^n)$, and so its highest weight is $k_1 e_1$. A highest weight vector is given by $(z_1 - iz_2)^{k_1} (z^\top z)^{k_2}$.
      \item For every $k_1, k_2 \in \N$, the subspace $\HH^{k_1}(\C^n) (z^\top z)^{k_2}$ is an irreducible module over $\SO(n)\times \SO(2)$ isomorphic to the outer tensor product $\HH^{k_1}(\C^n) \otimes \C_{k_1 + 2k_2}$, where $\C_{k_1 + 2k_2}$ is the $1$-dimensional $\SO(2)$-module defined by the character $\chi_{-(k_1+2k_2)}(t) = t^{-(k_1+2k_2)}$.
      \item The decomposition is multiplicity-free.
    \end{enumerate}
\end{proposition}
\begin{proof}
    First we note that, by the very definition of both $\SO(n)$ and $z^\top z$, every polynomial of the form $(z^\top z)^m$ is $\SO(n)$-invariant for every $m \in \N$. Hence, for every $k_1, k_2 \in \N$, the assignment
    \begin{align*}
      \HH^{k_1}(\C^n) &\rightarrow \HH^{k_1}(\C^n) (z^\top z)^{k_2} \\
      p(z) &\mapsto p(z) (z^\top z)^{k_2}
    \end{align*}
    clearly defines a linear map which is $\SO(n)$-equivariant for the representation $\pi_\lambda|_{\SO(n)}$ on both spaces. Since this is clearly an isomorphism of vector spaces we conclude that (1) holds.

    Consider the assignment
    \begin{align*}
      \HH^{k_1}(\C^n) (z^\top z)^{k_2} &\rightarrow \HH^{k_1}(\C^n) \otimes \C_{k_1 + 2k_2}  \\
      p(z) (z^\top z)^{k_2} &\mapsto p(z) \otimes 1.
    \end{align*}
    which is clearly a linear isomorphism. Since $(z^\top z)^{k_2}$ is $\SO(n)$-invariant, it follows that this linear map is an isomorphism of $\SO(n)$-modules. On other hand, for every $p(z) \in \HH^{k_1}(\C^n)$ and  $t \in \SO(2)$ we have
    \begin{align*}
      \pi_\lambda(I,t) (p(z)(z^\top z)^{k_2})
                &= t^{-(k_1 + 2k_2)} p(z)(z^\top z)^{k_2} \\
      (I,t)\cdot p(z) \otimes 1 &= p(z) \otimes t^{-(k_1+2k_2)} =
                t^{-(k_1+2k_2)}(p(z)\otimes 1),
    \end{align*}
    and so the last assignment is $\SO(2)$-equivariant as well. We conclude that such assignment is an isomorphism of $\SO(n)\times\SO(2)$-modules. This proves (2).

    By Problem 11 in page 271 from \cite{KnappBeyond} we have a direct sum decomposition of vector spaces
    \[
        \cP^m(\C^n) = \bigoplus_{k_1 + 2k_2 = m} \HH^{k_1}(\C^n) (z^\top z)^{k_2}
    \]
    for every $m \in \N$. Furthermore, by (2) this is a direct sum of mutually non-isomorphic irreducible $\SO(n)\times\SO(2)$-modules, and so the direct sum is orthogonal for the inner product of any Bergman space $\cA^2_\lambda(\DIV{n})$, for $\lambda > n-1$. By the density of the polynomials in Bergman spaces, we have the required Hilbert direct sum decomposition
    \[
        \cA^2_\lambda(\DIV{n})
            = \bigoplus_{m=0}^{\infty}
                \bigoplus_{\substack{k_1, k_2 \in \N \\ k_1 + 2k_2 = m}} 
                		\HH^{k_1}(\C^n) (z^\top z)^{k_2}.
    \]
    Finally, it follows from (2) that any two terms of such sum are mutually non-isomorphic over $\SO(n)\times\SO(2)$. This shows (3) and completes the proof.
\end{proof}

We remark that the information about highest weight vectors given above can also be found in \cite{Johnson}.

\subsection{Isotypic decomposition for $\SO(n-1)\times \SO(2)$.}
Let us consider the canonical upper left corner embedding of $\SO(n-1)$ as a subgroup of $\SO(n)$. In particular, the group $\SO(n-1) \times \SO(2)$ has a unitary representation on every Bergman space $\cA^2_\lambda(\DIV{n})$ by restriction of $\pi_\lambda$. The Hilbert direct sum decomposition of $\cA^2_\lambda(\DIV{n})$ from Proposition~\ref{prop:isoSO(n)SO(2)} is $\SO(n)\times\SO(2)$-invariant and so it is $\SO(n-1)\times\SO(2)$-invariant as well. However, the terms of such decomposition may no longer be irreducible when considered as $\SO(n-1) \times \SO(2)$-modules, but they will rather have a direct sum decomposition into irreducible $\SO(n-1) \times \SO(2)$-submodules.

For the subgroup $\SO(n-1) \subset \SO(n)$ it is well known how to decompose every irreducible $\SO(n)$-module as a direct sum of irreducible $\SO(n-1)$-submodules. Statements that provide such decompositions are known as branching rules, and they are often given by describing the submodules in terms of highest weights. For our case, the most general classical result is the Branching Theorem of Murnaghan (see Theorem 9.16 from \cite{KnappBeyond}). A particular case of such branching theorem together with our Proposition~\ref{prop:harmonic-irred} yields the next result.

\begin{proposition}\label{prop:branching}
	Let $n \geq 4$ be given. Then, for every $m \in \N$, the subspace $\HH^m(\C^n)$ considered as a module over $\SO(n-1)$ satisfies
	\[
		\HH^m(\C^n) \simeq \bigoplus_{j=0}^m \HH^j(\C^{n-1})
	\]
	where each space $\HH^j(\C^{n-1})$ is considered as a $\SO(n-1)$-module and the isomorphism holds over $\SO(n-1)$. Hence, this decomposition is multiplicity-free.
\end{proposition}

We can now use Propositions \ref{prop:isoSO(n)SO(2)} and \ref{prop:branching} to obtain the isotypic decomposition for the group $\SO(n-1)\times \SO(2)$. In what follows we will denote by $\lfloor m\rfloor_2$ the parity of $m \in \N$, where the latter is defined as either $0$ or $1$ according to whether $m$ is even or odd, respectively. We will also use the expression $k V$ to denote the direct sum of $k$ copies of $V$, where $k \in \N$ and $V$ is an irreducible module over some compact group. As before, $\C_m$ denotes the $1$-dimensional $\SO(2)$-module corresponding to the character $\chi_{-m}(t) = t^{-m}$.

\begin{proposition}\label{prop:isoSO(n-1)SO(2)}
	Let $n \geq 4$ be given. Then, for every $\lambda > n-1$, the isotypic decomposition of $\pi_\lambda|_{\SO(n-1)\times\SO(2)}$ is given by the following isomorphism of modules over $\SO(n-1) \times \SO(2)$.
	\[
		\cA^2_\lambda(\DIV{n}) \simeq 
			\bigoplus_{m=0}^\infty \bigoplus_{r=0}^m
				\bigg(\bigg\lfloor \frac{m}{2} \bigg\rfloor - 
						\bigg\lfloor  
						\frac{r + \lfloor m+1 \rfloor_2}{2}
						\bigg\rfloor + 1
				\bigg)
					\HH^r(\C^{n-1}) \otimes \C_m.
	\]
	In particular, this decomposition is not multiplicity-free.
\end{proposition}
\begin{proof}
	The isotypic decomposition required is necessarily $\SO(2)$-invariant, and so it is some refinement of the decomposition obtained in Proposition~\ref{prop:isoSO(2)}. In particular, it is enough to obtain the isotypic decomposition of $\cP^m(\C^n)$ by restricting from $\SO(n)$ to $\SO(n-1)$, for which we will use Proposition~\ref{prop:branching}.
	
	First we can use the proof of Propositions~\ref{prop:isoSO(n)SO(2)} to conclude that for every $m \in \N$ we have the following identities and isomorphisms of modules over $\SO(n-1) \times \SO(2)$
	\begin{align*}
		\cP^m(\C^n) &= \bigoplus_{\substack{k_1, k_2 \in \N \\ k_1 + 2k_2 = m}} 
		\HH^{k_1}(\C^n) (z^\top z)^{k_2} \\
		&\simeq \bigoplus_{k=0}^{\lfloor m/2\rfloor} \HH^{m-2k}(\C^n) \otimes \C_m \\
		&\simeq \bigoplus_{k=0}^{\lfloor m/2\rfloor}
			\bigoplus_{r=0}^{m-2k} \HH^r(\C^{n-1}) \otimes \C_m,
	\end{align*}
	where we have used Proposition~\ref{prop:branching} in the last isomorphism. The isotypic decomposition is now obtained by counting how many times each $\HH^r(\C^{n-1}) \otimes \C_m$ appears.
	
	For $m = 2\ell$ even, there are $\ell + 1$ terms in the second line above, and these are: $\HH^{2\ell}(\C^n) \otimes \C_m, \dots, \HH^0(\C^n) \otimes \C_m$. Each such term contains a copy of $\HH^0(\C^{n-1}) \otimes \C_m$, for a total of $\ell+1$ occurrences. But only the first $\ell$ terms contain copies of $\HH^1(\C^{n-1}) \otimes \C_m$ and $\HH^2(\C^{n-1}) \otimes \C_m$, which gives $\ell$ occurrences. In general, for every $j \geq 1$, there is a single copy of each of the terms $\HH^{2j-1}(\C^{n-1}) \otimes \C_m$ and  $\HH^{2j}(\C^{n-1}) \otimes \C_m$ inside each of $\HH^{2\ell}(\C^n) \otimes \C_m, \dots, \HH^{2j}(\C^n) \otimes \C_m$, which yields $\ell - j + 1$ occurrences. We now observe that for $r = 2j, 2j-1$ we have
	\[
		j = \bigg\lfloor 
			\frac{r + 1}{2}
			\bigg\rfloor
			= \bigg\lfloor 
			\frac{r + \lfloor m+1 \rfloor_2}{2}
			\bigg\rfloor
	\]
	since $m = 2\ell$ is even. Hence, the result follows in this case.
	
	For the case $m = 2\ell + 1$, a similar argument shows that, for every $j \geq 0$, the terms $\HH^{2j}(\C^{n-1}) \otimes \C_m$ and  $\HH^{2j+1}(\C^{n-1}) \otimes \C_m$ have exactly $\ell - j + 1$ occurrences. Since now $m = 2\ell + 1$ is odd, we now have for $r = 2j, 2j+1$ that
	\[
		j = \bigg\lfloor 
			\frac{r}{2}
			\bigg\rfloor
			= \bigg\lfloor 
				\frac{r + \lfloor m+1 \rfloor_2}{2}
				\bigg\rfloor.
	\]
	Hence, the result follows as well in this case.
\end{proof}

\subsection{Isotypic decomposition for $\mrT{n} \times \SO(2)$.}
Recall that $\mrT{n} \times \SO(2)$ denotes the maximal torus of $\SO(n) \times \SO(2)$ introduced in subsection~\ref{subsec:isotypicSOnSO2}, which has dimension $\lfloor n/2\rfloor +1$. Its irreducible representations are $1$-dimensional and completely described in terms of characters. As before, the isotypic decomposition for the group $\mrT{n} \times \SO(2)$ is a refinement of the one given in Proposition~\ref{prop:isoSO(2)}. Furthermore, it is enough to obtain the decomposition of each of the subspaces $\cP^m(\C^n)$ with respect to the $\mrT{n}$-action.

The subspace $\cP^m(\C^n)$ has a natural basis given by the monomials $z^\alpha$ where $\alpha \in \N^n$ is such that $|\alpha| = m$. We will replace this basis by another obtained from a linear change of coordinates. More precisely, we have the following basis for every subspace $\cP^m(\C^n)$.

\begin{enumerate}
	\item If $n = 2\ell$ is even, then a basis for $\cP^m(\C^n)$ is given by the polynomials
	\[
		q_{\alpha,\beta}(z)  =
		\prod_{j=1}^{\ell} (z_{2j-1} - i z_{2j})^{\alpha_j} 
				(z_{2j-1} + i z_{2j})^{\beta_j}
	\]
	where $\alpha, \beta \in \N^\ell$ satisfy $|\alpha| + |\beta| = m$.
	\item If $n = 2\ell+1$ is odd, then a basis for $\cP^m(\C^n)$ is given by the polynomials
	\[
		q_{\alpha,\beta,\gamma}(z)  =
		\prod_{j=1}^{\ell} (z_{2j-1} - i z_{2j})^{\alpha_j} 
			(z_{2j-1} + i z_{2j})^{\beta_j} z_{2\ell+1}^\gamma
			= q_{\alpha,\beta}(z') z_{2\ell+1}^\gamma
	\]
	where $z'$ is obtained from $z$ by removing the last coordinate, and $\alpha, \beta \in \N^\ell, \gamma \in \N$ satisfy $|\alpha| + |\beta| + \gamma= m$.	
\end{enumerate}

A straightforward computation (see the proof of Proposition~\ref{prop:harmonic-irred} as well as page 277 from \cite{KnappBeyond}) shows that for $n = 2\ell$ and every $\theta \in \R^\ell$ we have
\[
	A(\theta) \cdot q_{\alpha,\beta}(z) = e^{i(\beta-\alpha)\cdot\theta} q_{\alpha,\beta}(z)
\]
for every $\alpha, \beta \in \N^\ell$, where $A(\theta) \in \mrT{n}$ is given as in equation~\eqref{eq:Atheta}. We conclude that $q_{\alpha,\beta}(z)$ is a weight vector for the $\mrT{n}$-action corresponding to the character $\chi_{\beta-\alpha}(t) = t^{\beta-\alpha}$ under the natural isomorphism $\mrT{n} \simeq \T^\ell$ given by the identification~\eqref{eq:TSO(2)}. From the above remarks it follows that for the case of $n = 2\ell +1$ a corresponding property holds as well. More precisely, $q_{\alpha,\beta,\gamma}(z)$ is a weight vector for the $\mrT{n}$-action corresponding to the character $\chi_{\beta-\alpha}(t) = t^{\beta-\alpha}$, again for the natural isomorphism $\mrT{n} \simeq \T^\ell$.

On the other hand, $\SO(2)$ acts on every subspace $\cP^m(\C^n)$ by the character $\chi_{-m}(t) = t^{-m}$. Hence, we obtain the next result. Recall our assumption $n \geq 3$.

\begin{proposition}\label{prop:isoTSO(2)}
	Let us identify the maximal torus $\mrT{n}\times \SO(2)$ of $\SO(n)\times\SO(2)$ with $\T^{\ell+1}$ through the isomorphism obtained from \eqref{eq:TSO(2)}, where $\ell = \lfloor n/2 \rfloor$. Then, for every $\lambda > n-1$, the following decompositions into irreducible modules over $\mrT{n} \times \SO(2)$ hold.
	\begin{enumerate}
		\item For $n = 2 \ell$ even we have
		\[
			\cA^2_\lambda(\DIV{n}) = 
				\bigoplus_{m=0}^\infty
				\bigoplus_{\substack{\alpha,\beta \in \N^\ell \\ |\alpha| + |\beta| = m}}
				\C q_{\alpha, \beta}(z),
		\]
		where two terms $\C q_{\alpha, \beta}(z)$ and $\C q_{\alpha', \beta'}(z)$ are isomorphic over $\mrT{n} \times \SO(2)$ if and only if $|\alpha| + |\beta| = |\alpha'| + |\beta'|$ and $\alpha - \beta = \alpha' - \beta'$.
		\item For $n = 2 \ell + 1$ odd we have
		\[
			\cA^2_\lambda(\DIV{n}) = 
				\bigoplus_{m=0}^\infty
				\bigoplus_{\substack{\alpha,\beta \in \N^\ell, \gamma \in \N \\ |\alpha| + |\beta| + \gamma = m}}
			\C q_{\alpha, \beta, \gamma}(z),
		\]
		where two terms $\C q_{\alpha, \beta, \gamma}(z)$ and $\C q_{\alpha', \beta', \gamma'}(z)$ are isomorphic over $\mrT{n} \times \SO(2)$ if and only if $|\alpha| + |\beta| + \gamma = |\alpha'| + |\beta'| + \gamma'$ and $\alpha - \beta = \alpha' - \beta'$.
	\end{enumerate}
	Hence, for both cases the isotypic decompositions are not multiplicity-free.
\end{proposition}
\begin{proof}
	The direct sum decomposition and the $\mrT{n} \times \SO(2)$-invariance of the terms was proved above. The characterization of the isomorphism type also follows from the previous discussion since two terms are isomorphic over $\mrT{n} \times \SO(2)$ if and only if they are isomorphic over both $\mrT{n}$ and $\SO(2)$. The last claim follows from the assumption $n \geq 3$.
\end{proof}

\begin{remark}\label{rmk:n=3_to_n-1}
	We observe that for $n = 3$, the maximal torus $\mrT{3}$ of $\SO(3)$ is the subgroup $\SO(2)$ with its canonical upper-left corner embedding. Hence, the case $n=3$ not considered in Proposition~\ref{prop:isoSO(n-1)SO(2)} is now included as part of  Proposition~\ref{prop:isoTSO(2)}.
\end{remark}

\subsection{$C^*$-algebras generated by Toeplitz operators with invariant symbols}\label{subsec:invariant}
For every closed subgroup $H \subset \SO(n)\times \SO(2)$ there is a natural action of $H$ on $L^\infty(\DIV{n})$ given by
\[
	h \cdot a = a \circ h^{-1},
\]
where $a \in L^\infty(\DIV{n})$ and $h \in H$. A symbol $a \in L^\infty(\DIV{n})$ is called $H$-invariant if it is fixed with respect to this action. We will denote by $L^\infty(\DIV{n})^H$ the subspace of all essentially bounded $H$-invariant symbols. Note that this subspace is clearly self-adjoint.

The next result follows from a straightforward computation (see \cite{DOQJFA}) and our previous definitions.

\begin{proposition}\label{prop:HinvHinter}
	Let $H \subset \SO(n)\times\SO(2)$ be a closed subgroup. Then, for every $\lambda > n-1$ and $a \in L^\infty(\DIV{n})$ we have
	\[
		T^{(\lambda)}_{h\cdot a} = 
			\pi_\lambda(h) \circ T^{(\lambda)}_a \circ \pi_\lambda(h)^{-1}
	\]
	for every $h \in H$. In particular, the symbol $a$ is $H$-invariant if and only if $T^{(\lambda)}_a$ intertwines the representation $\pi_\lambda|_H$. Furthermore, we have
	\[
		\cT^{(\lambda)}(L^\infty(\DIV{n})^H) \subset \End_H(\cA^2_\lambda(\DIV{n})),
	\]
	for every $\lambda > n-1$.
\end{proposition}

We now obtain a result that can be considered as providing a correspondence between multiplicity-free restrictions of $\pi_\lambda$ and commutative $C^*$-algebras generated by Toeplitz operators with invariant symbols. This result is similar to Theorem~6.4 from \cite{DOQJFA}, but we require an additional condition which is enough for our purposes. We provide a sketch of the proof, based on the arguments from \cite{DOQJFA}, for the sake of completeness.

\begin{theorem}\label{thm:multfree-comm}
	Let $H \subset \SO(n)\times\SO(2)$ be a closed subgroup such that the isotypic components of $\pi_\lambda|_H$ are finite dimensional. Then, for every $\lambda > n-1$ the following conditions are equivalent.
	\begin{enumerate}
		\item The restriction $\pi_\lambda|_H$ is multiplicity-free.
		\item The von Neumann algebra $\End_H(\cA^2_\lambda(\DIV{n}))$ is commutative.
		\item The $C^*$-algebra $\cT^{(\lambda)}(L^\infty(\DIV{n})^H)$ is commutative.
	\end{enumerate}
\end{theorem}
\begin{proof}
	Consider the isotypic decomposition for $\pi_\lambda|_H$ given by
	\[
		\cA^2_\lambda(\DIV{n}) = \bigoplus_{j\in A} \HH_j,
	\]
	where $\HH_j$ is an isotypic component. 
	
	If (1) holds, then each subspace $\HH_j$ is irreducible over $H$. An application of Schur's lemma shows that for every $T \in \End_H(\cA^2_\lambda(\DIV{n}))$ we have $T(\HH_j) \subset \HH_j$ for every $j \in A$. Furthermore, we can also conclude that for such a $T$ we have $T|_{\HH_j} = c_j I_{\HH_j}$ for some $c_j \in \C$ and every $j \in A$. Hence, the commutativity of $\End_H(\cA^2_\lambda(\DIV{n}))$ is immediate, thus showing (2). On the other hand, (2) clearly implies (3) by Proposition~\ref{prop:HinvHinter}.
	
	If the isotypic decomposition above is not multiplicity-free, then for some $j \in A$ there exists $U,V \subset \HH_j$ isomorphic irreducible $H$-modules such that $U\cap V = 0$. If we choose $T_0 : U \rightarrow V$ any isomorphism over $H$, then for any matrix $A \in M_{2\times2}(\C)$ with entries
	\[
		A = \begin{pmatrix}
				a & b \\
				c & d		
			\end{pmatrix}
	\]
	the map given by
	\begin{align*}
		\widehat{T}_A : U\oplus V &\rightarrow U\oplus V \\
		u+v &\mapsto au + bT_0^{-1}(v) + cT_0(u) + dv,
	\end{align*}
	where $u \in U$ and $v \in V$, intertwines the $H$-action. This map can be extended to an element $T_A \in \End_H(\cA^2_\lambda(\DIV{n}))$ as $0$ on the orthogonal complement of $U \oplus V$. It is easy to see that the assignment $A \mapsto T_A$ defines an injective homomorphism of algebras from $M_{2\times2}(\C)$ into $\End_H(\cA^2_\lambda(\DIV{n}))$, thus showing that the latter is not commutative. This proves that (2) implies (1).

	Finally, to prove that (3) implies (1), we assume that $\pi_\lambda|_H$ is not multiplicity-free. Hence, in the notation above, there exists an isotypic component $\HH_j$ which is not irreducible but, by assumption, it is finite dimensional. In particular, for some $m \geq 2$ we have $\End_H(\HH_j) \simeq M_{m\times m}(\C)$ as $C^*$-algebras. Choose some non-normal element $M_{m\times m}(\C)$, denote by $T$ the corresponding element in $\End_H(\HH_j)$ and extend it by $0$ on the rest of the isotypic components to obtain a non-normal element in $\End_H(\cA^2_\lambda(\DIV{n}))$ which we will also denote by $T$. Using the averaging techniques of Section~6 from \cite{DOQJFA} we can find a symbol $a \in L^\infty(\DIV{n})^H$ such that
	\[
		\langle T f, g \rangle_\lambda = 
			\langle T^{(\lambda)}_a f, g \rangle_\lambda
	\]
	for every $f, g \in \HH_j$ (see the proof of Proposition~6.2 from \cite{DOQJFA}). Since $\HH_j$ is invariant under any element from $\End_H(\cA^2_\lambda(\DIV{n}))$ it follows that
	\[
		T|_{\HH_j} = T^{(\lambda)}_a|_{\HH_j}, \quad
			T^*|_{\HH_j} = (T^{(\lambda)}_a)^*|_{\HH_j},
	\]
	thus proving that $T^{(\lambda)}_a \in \cT^{(\lambda)}(L^\infty(\DIV{n})^H)$ is not normal. Since $\cT^{(\lambda)}(L^\infty(\DIV{n})^H)$ is a $C^*$-algebra, it follows that it is not commutative.	Hence, (3) implies (1).
\end{proof}

We apply the previous criterion to the subgroups of $\SO(n)\times \SO(2)$ whose isotypic decompositions were computed in the previous subsections. More precisely, the next result is a consequence of Theorem~\ref{thm:multfree-comm}, Propositions~\ref{prop:isoSO(2)}, \ref{prop:isoSO(n)SO(2)}, \ref{prop:isoSO(n-1)SO(2)} and \ref{prop:isoTSO(2)} and Remark~\ref{rmk:n=3_to_n-1}.

\begin{theorem}\label{thm:comm-invariance}
	Assume that $n \geq 3$. Then, for every $\lambda > n-1$, and for the subgroups specified, the following commutativity properties hold.
	\begin{enumerate}
		\item The $C^*$-algebra $\cT^{(\lambda)}(L^\infty(\DIV{n})^{\SO(n) \times \SO(2)})$ is commutative.
		\item For $H$ either of the subgroups $\SO(n-1) \times \SO(2)$, $\mrT{n} \times \SO(2)$ or $\SO(2)$, the $C^*$-algebra $\cT^{(\lambda)}(L^\infty(\DIV{n})^H)$ is not commutative.
	\end{enumerate}
\end{theorem}

\begin{remark}\label{rmk:noncommMASG}
	If $H_1 \subset H_2$ are subgroups of $\SO(n)\times\SO(2)$, then we clearly have
	\[
		L^\infty(\DIV{n})^{H_2} \subset L^\infty(\DIV{n})^{H_1}
	\]
	and correspondingly we have the inclusion
	\[
		\cT^{(\lambda)}(L^\infty(\DIV{n})^{H_2}) \subset \cT^{(\lambda)}(L^\infty(\DIV{n})^{H_1}).
	\]
	In other words, both the invariance of symbols and the $C^*$-algebra generated by Toeplitz operators with invariant symbols reverse inclusion with respect to the subgroup considered. In Theorem~\ref{thm:comm-invariance}, the largest subgroup considered is $\SO(n)\times\SO(2)$, which yields the smallest and only (in the statement) commutative $C^*$-algebra through invariant symbols. The smallest subgroup in that result is $\SO(2)$ for which we obtain the largest of the $C^*$-algebras considered and it is non-commutative. 
	
	The subgroups in between, $\SO(n-1) \times \SO(2)$ and $\mrT{n} \times \SO(2)$, provide invariant symbols that yield non-commutative $C^*$-algebras. And there are some interesting facts about these cases.
	
	The subgroup $\mrT{n} \times \SO(2)$ is a maximal Abelian subgroup of the biholomorphism group of $\DIV{n}$. As noted in \cite{QVUnitBall1,QVUnitBall2}, every maximal Abelian subgroup of the biholomorphism group of the  unit ball $\mathbb{B}^n$ yields invariant symbols whose Toeplitz operators generate commutative $C^*$-algebras. Hence, the previously observed behavior for the unit ball does not extend to the domain $\DIV{n}$. We recall that the unit ball has rank $1$ and the domain $\DIV{n}$ has rank $2$ for every $n \geq 2$.
\end{remark}

\begin{lemma}\label{lem:maxSO}
	For every $n \geq 3$, the subgroup $\SO(n-1) \times \SO(2)$ is a maximal connected subgroup of $\SO(n) \times \SO(2)$.
\end{lemma}
\begin{proof}
	We note that it is enough to show that $\SO(n-1)$ is a maximal connected subgroup of $\SO(n)$ whenever $n \geq 3$. In the notation of \cite{HelgasonDGLSS}, and for the embedding $\SO(n-1) \subset \SO(n)$ considered, the pair $(\SO(n), \SO(n-1))$ is an irreducible Riemannian symmetric pair since the corresponding quotient yields the sphere $S^{n-1}$, an irreducible symmetric space of compact type for $n \geq 3$. Hence, Proposition~5.1 in Chapter~VIII from \cite{HelgasonDGLSS} implies the maximality of $\SO(n-1)$ in $\SO(n)$.
\end{proof}

As a consequence of Theorem~\ref{thm:comm-invariance} and Lemma~\ref{lem:maxSO} we obtain the following result.

\begin{corollary}\label{cor:noncommMaxConn}
	For $n \geq 3$, the subgroup $\SO(n-1) \times \SO(2)$ is a maximal connected subgroup of $\SO(n) \times \SO(2)$ such that $\cT^{(\lambda)}(L^\infty(\DIV{n})^{\SO(n) \times \SO(2)})$ is commutative but $\cT^{(\lambda)}(L^\infty(\DIV{n})^{\SO(n-1) \times \SO(2)})$ is non-commutative, for every $\lambda > n-1$.
\end{corollary}

\begin{remark}\label{rmk:noncommMaxConn}
	As noted in Remark~\ref{rmk:noncommMASG}, there exist maximal Abelian subgroups of the group of biholomorphisms of $\DIV{n}$ whose corresponding invariant symbols yield Toeplitz operators that generate non-commutative $C^*$-algebras. In fact, the subgroup $\mrT{n} \times \SO(2)$ is maximal Abelian in $\SO(n) \times \SO(2)$, and we proved that, through invariance, the latter yields commutative $C^*$-algebras, but the former does~not.
	
	We might try to relax the conditions to obtain commutativity. For example, instead of considering maximal Abelian subgroups we may consider maximal connected subgroups, whether Abelian or not. However, Corollary~\ref{cor:noncommMaxConn} proves the existence of a maximal connected subgroup of $\SO(n) \times \SO(2)$ whose invariant symbols yield Toeplitz operators that generate a non-commutative $C^*$-algebra. 
	
	On the other hand, we note that $\mrT{n}$, our maximal torus for $\SO(n)$, is contained in $\SO(n-1)$ if and only if $n$ is odd. It is precisely in this case that we have the inclusion $\mrT{n} \times \SO(2) \subset \SO(n-1) \times \SO(2)$. Hence, for $n$ odd, $\SO(n-1) \times \SO(2)$ is a maximal connected subgroup of the group $\SO(n) \times \SO(2)$ that furthermore contains a maximal toral subgroup. So not even both of these conditions, are enough to yield commutative $C^*$-algebras.
	
	Note however, that we have made no attempt to list the maximal connected subgroups of $\SO(n) \times \SO(2)$. But, we did exhibit such a subgroup for which we obtain a non-commutative $C^*$-algebra generated by Toeplitz operators. The problem as to whether some other maximal connected subgroup of $\SO(n)\times \SO(2)$ yields commutative $C^*$-algebras through invariance remains open in general.
\end{remark}
\begin{remark}\label{rmk:n=4DI22}
	In dimension $n = 4$, there is one known case for which we do obtain commutative $C^*$-algebras generated by Toeplitz operators from invariance of symbols with respect to a maximal connected subgroup of the isotropy group of the origin. This is a consequence of the results from \cite{DOQMatrix22} and we now proceed to explain it together with its relation to our current work. 
	
	We recall from Section~\ref{sec:DIV} that the domain $\DIV{4}$ is biholomorphically equivalent to $\mathrm{D}^\mathrm{I}_{2 \times 2}$, in accordance to the isomorphism $\so(4,2) \simeq \su(2,2)$. For the domain $\mathrm{D}^\mathrm{I}_{2\times 2}$, the maximal compact subgroup of biholomorphisms that fix the origin is realized by the group $\Spe(\U(2) \times \U(2))$. With the diagonal embedding $\T^2 \subset \U(2)$, a natural maximal toral subgroup of $\Spe(\U(2) \times \U(2))$ is given by $\Spe(\T^2 \times \T^2)$. It is known (see~\cite{DOQJFA}) that, through the use of invariant symbols, the former group yields commutative $C^*$-algebras generated by Toeplitz operators, while the latter subgroup yields non-commutative ones. However, it was proved in \cite{DOQMatrix22} that the symbols invariant under the subgroup $\Spe(\U(2) \times \T^2)$ yield Toeplitz operators generating commutative $C^*$-algebras. Furthermore, $\Spe(\U(2) \times \T^2)$ is easily seen to be a maximal connected subgroup of the isotropy group at the origin $\Spe(\U(2) \times \U(2))$ (see \cite{DOQMatrix22}) and it clearly contains the maximal toral subgroup $\Spe(\T^2 \times \T^2)$. Through the biholomorphism between $\DIV{4}$ and $\mathrm{D}^\mathrm{I}_{2 \times 2}$ this implies the existence of a maximal connected subgroup $H$ of $\SO(4) \times \SO(2)$ containing a maximal toral subgroup so that the $H$-invariant symbols yield Toeplitz operators generating commutative $C^*$-algebras. This is in clear contrast with Corollary~\ref{cor:noncommMaxConn}, so we explain the nature of this subgroup $H$.
	
	The subgroup $H$ can be described as follows through its Lie algebra. For $n = 4$, we have the well known isomorphisms
	\[
		\so(4) \simeq \so(3) \oplus \so(3) \simeq \su(2) \oplus \su(2).
	\]
	Hence, by using our identification $\SO(2) \simeq \T$, the Lie algebras of the isotropy subgroups at the origin of $\DIV{4}$ and $\mathrm{D}^\mathrm{I}_{2 \times 2}$ are related by the isomorphism
	\[
		\so(4) \oplus \so(2) \simeq \so(3) \oplus \so(3) \oplus \so(2) \simeq \su(2) \oplus \su(2) \oplus \so(2),
	\]
	and the subgroup $\Spe(\U(2) \times \T^2)$ has Lie subalgebra $\su(2) \oplus \so(2) \oplus \so(2)$ naturally embedded in the third term of the last sequence of isomorphisms. Hence, through such isomorphisms, the maximal connected subgroup $H \subset \SO(4) \times \SO(2)$ has Lie subalgebra $\fh$ given by 
	\[
		\fh = \so(3) \oplus \so(2) \oplus \so(2) \subset \so(3) \oplus \so(3) \oplus \so(2) \simeq \so(4) \oplus \so(2).
	\]
	In particular, $H$ is locally isomorphic to $\SO(3) \times \SO(2) \times \SO(2)$.
	
	On the other hand, the subgroup $\SO(3) \times \SO(2) \subset \SO(4) \times \SO(2)$ considered in Corollary~\ref{cor:noncommMaxConn} (and throughout this work in general) corresponds, at the Lie algebra level, to the diagonal embedding
	\begin{align*}
		\so(3) &\rightarrow \so(3) \oplus \so(3) \simeq \so(4) \\
		X &\mapsto (X,X).
	\end{align*}
	This is a consequence of the local equivalence of the pairs of groups $(\SO(4), \SO(3))$ and $(\SO(3)\times\SO(3), \SO(3))$ (for the diagonal embedding in the latter) considered as Riemannian symmetric pairs. Such local equivalence is explained in Example~II in page~240 from~\cite{HelgasonDGLSS}.
	
	Hence, the fact that $\SO(4)$ is not simple allows to have enough room to embed $\SO(3)$ in an alternative manner to obtain a maximal connected subgroup $H$ of $\SO(4) \times \SO(2)$, locally isomorphic to $\SO(3) \times \SO(2) \times \SO(2)$, containing a maximal toral subgroup, and so that the $C^*$-algebras $\cT^{(\lambda)}(L^\infty(\DIV{4})^H)$ are commutative for every $\lambda > 3$. This fact further emphasizes our claim in Remark~\ref{rmk:noncommMaxConn} that it remains open to consider all maximal connected subgroups of $\SO(n) \times \SO(2)$.
\end{remark}

\section{Moment maps of the torus $\mrT{n}\times\SO(2)$ and its subgroup $\SO(2)$}\label{sec:momentmaps}
In this section we will introduce a type of symbol that can be associated to actions using the symplectic structure of the domain $\DIV{n}$. This is achieved through the notion of moment map, which we will compute for the maximal torus introduced before.

We start by considering a Lie subgroup $H$ of the biholomorphism group of $\DIV{n}$ and by $\fh$ its corresponding Lie algebra. Then, for every $X \in \fh$ there is an associated $1$-parameter subgroup of biholomorphisms $r \mapsto \exp(rX)$ whose orbits in $\DIV{n}$ we can differentiate to obtain a vector field on $\DIV{n}$ as follows
\[
	X^\sharp_z = \eval[2]{\frac{\dif}{\dif r}}_{r=0}
		\exp(rX)\cdot z,
\]
for every $z \in \DIV{n}$. It is a straightforward exercise to show that if $X^\sharp = (g_1, \dots, g_n)$ as a $\C^n$-valued function, then we have
\begin{equation}\label{eq:Xsharpinpartials}
	X^\sharp = \sum_{j=1}^n \bigg(g_j \frac{\partial}{\partial z_j}
		+ \overline{g}_j \frac{\partial}{\partial \overline{z}_j}
	\bigg).
\end{equation}

Recall that $\DIV{n}$ is a K\"ahler manifold with symplectic form $\omega$ given by Proposition~\ref{prop:SympForm}. We also recall that for a given Lie group $H$ with Lie algebra $\fh$, the adjoint representation of $H$ is a homomorphism $\Ad = \Ad_H : H \rightarrow \GL(\fh)$ obtained by differentiating the conjugation (see \cite{HelgasonDGLSS}). This induces a representation $\Ad^* : H \rightarrow \GL(\fh^*)$ on the dual vector space $\fh^*$ given by $\Ad^*(h) = \Ad(h^{-1})^*$, the transpose or dual of $\Ad(h^{-1})$. Also, we will denote by $\langle\cdot,\cdot\rangle$ the dual evaluation form between $\fh$ and $\fh^*$. We also recall that, from the remarks at the end of subsection~\ref{subsec:Bergmanmetric}, if a group $H$ acts biholomorphically on $\DIV{n}$, then the $H$-action preserves the symplectic form $\omega$ of this domain. We now recall the definition of the moment map of a symplectic action for our setup. We refer to \cite{McDuffSalamon} for further details in the general case.

\begin{definition}\label{def:momentmap}
	Let $H$ be a Lie subgroup of the biholomorphism group of the domain $\DIV{n}$, whose Lie algebra is denoted by $\fh$. A moment map for the $H$-action is a smooth map $\mu = \mu^H : \DIV{n} \rightarrow \fh^*$ for which the following properties hold.
	\begin{enumerate}
		\item For every $X \in \fh$, the function $\mu_X : \DIV{n} \rightarrow \R$ given by 
		\[
			\mu_X(z) = \langle \mu(z), X\rangle
		\]
		satisfies $\dif \mu_X(u) = \omega(X^\sharp, u)$ for every $u$ tangent to $\DIV{n}$. 
		\item The function $\mu$ is $H$-equivariant. In other words, we have
		\[
			\mu(h\cdot z) = \Ad^*(h)(\mu(z)),
		\]
		for every $h \in H$ and $z \in \DIV{n}$.
	\end{enumerate}
\end{definition}

For a given smooth function $f : \DIV{n} \rightarrow \R$, the Hamiltonian vector field of $f$ is the smooth vector field $X_f$ that satisfies
\[
	\dif f(u) = \omega(X_f, u)
\]
for every tangent vector $u$. This definition generalizes in an obvious manner to any symplectic manifold. Hence, the first condition given in Definition~\ref{def:momentmap} requires the Hamiltonian vector field of $\mu_X$ to be $X^\sharp$, for every $X \in \fh$.

We note that if $H$ is Abelian, then $\Ad$ is the trivial homomorphism and so the second condition in Definition~\ref{def:momentmap} reduces to requiring the $H$-invariance of $\mu$ in this case. Furthermore, for $H$ Abelian, the Lie algebra $\fh$ is isomorphic to $\R^k$ where $k = \dim H$. And we can identify $\fh = \R^k$ through the choice of some basis so that $\langle\cdot,\cdot\rangle$ is precisely the inner product obtained from the chosen basis.

Let us apply the previous constructions to the maximal torus $\mrT{n} \times \SO(2)$. Its Lie algebra is given by $\ft_0 \times \so(2)$ and the Lie algebra $\ft_0$ and its elements have been described before in \eqref{eq:ft0}. Such description provides a natural basis for $\ft_0$. More precisely, for $\ell = \lfloor n/2 \rfloor$ and every $j = 1, \dots, \ell$, let us denote
\[
	X_j = \mathrm{diag}(0, \dots, R(\pi/2), \dots, \dots)
\]
the block diagonal matrix with blocks of size $2 \times 2$, with $R(\pi/2)$ appearing in the $j$-th position and a trailing $0$ at the end when $n$ is odd. As before $R(\pi/2)$ denotes the matrix defined in \eqref{eq:Rpi2}. We also consider the matrix
\[
	X_{\ell + 1} = R(\pi/2) \in \SO(2).
\]
Then, the elements $X_1, \dots, X_{\ell + 1}$ yield a basis for the Lie algebra $\ft_0 \times \so(2)$. In the rest of this work we will identify $\ft_0 \times \so(2) = \R^{\ell + 1}$ through this basis. In particular, we will refer to $X_1, \dots, X_{\ell + 1}$ as the canonical basis of $\ft_0 \times \so(2)$. Note that this yields the identification between $\ft_0 \times \so(2)$ and its dual space described above.

We proceed to compute the moment map for the action of $\mrT{n} \times \SO(2)$ on $\DIV{n}$. We start by proving an easy equivalent form to obtain a moment map in this case.

\begin{lemma}\label{lem:momentmapbasis}
	Let $\ell = \lfloor n/2\rfloor$, and consider the identification $\ft_0 \times \so(2) = \R^{\ell + 1}$ obtained from its canonical basis. For a collection of real-valued functions $f_1, \dots, f_{\ell+1}$ defined on $\DIV{n}$ consider the map $\mu : \DIV{n} \rightarrow \R^{\ell + 1}$ given by
	\[
		\mu = (f_1, \dots, f_{\ell+1}) = \sum_{j=1}^{\ell+1} f_j X_j.
	\]
	Then, $\mu$ is a moment map for the action of $\mrT{n} \times \SO(2)$ if and only if the following conditions are satisfied for every $j = 1, \dots, \ell + 1$
	\begin{enumerate}
		\item $\dif f_j = \omega(X^\sharp_j, \cdot)$,
		\item $f_j \circ t = f_j$ for every $t \in \mrT{n} \times \SO(2)$.
	\end{enumerate}
\end{lemma}
\begin{proof}
	Given our current identifications and from the definition of $\mu$ we have
	\[
		\mu_X = \sum_{j=1}^{\ell+1} \theta_j f_j
	\]
	for every $X = \theta_1 X_1 + \dots + \theta_{\ell+1} X_{\ell+1}$. Hence, conditions (1) and (2) from our statement are precisely the conditions from Definition~\ref{def:momentmap} for the canonical basis. Then, the claimed equivalence follows from the linear dependence of $\mu_X$ and $X^\sharp$ as functions of $X$.
\end{proof}

The next step is to compute the vector fields on $\DIV{n}$ associated to the elements of $\ft_0 \times \so(2)$.

\begin{lemma}\label{lem:Xsharp}
	Let $\ell = \lfloor n/2\rfloor$, and consider the identification $\ft_0 \times \so(2) = \R^{\ell + 1}$ obtained from its canonical basis. For every $X \in \ft_0 \times \so(2)$ of the form
	\[
		X = \sum_{j=1}^{\ell+1} \theta_j X_j
	\]
	the induced vector field $X^\sharp$ on $\DIV{n}$ is given by
	\begin{multline*}
		X^\sharp_z = \sum_{j=1}^{\ell}
			\theta_j
			\bigg(
				z_{2j} \frac{\partial}{\partial z_{2j-1}}
				-z_{2j-1} \frac{\partial}{\partial z_{2j}}
				+\overline{z}_{2j} \frac{\partial}{\partial \overline{z}_{2j-1}}
				-\overline{z}_{2j-1} \frac{\partial}{\partial \overline{z}_{2j}}
			\bigg) \\
			+ i \theta_{\ell+1}\sum_{j=1}^{n}
			\bigg(
				z_j \frac{\partial}{\partial z_j}
				-\overline{z}_j \frac{\partial}{\partial \overline{z}_j}
			\bigg),
	\end{multline*}
	for every $z \in \DIV{n}$. In particular, we have for every $j = 1, \dots, \ell$
	\[
		X^\sharp_j(z) = z_{2j} \frac{\partial}{\partial z_{2j-1}}
		-z_{2j-1} \frac{\partial}{\partial z_{2j}}
		+\overline{z}_{2j} \frac{\partial}{\partial \overline{z}_{2j-1}}
		-\overline{z}_{2j-1} \frac{\partial}{\partial \overline{z}_{2j}}
	\]
	and also
	\[
		X^\sharp_{\ell+1}(z) =
		i \sum_{j=1}^{n}
		\bigg(
		z_j \frac{\partial}{\partial z_j}
		-\overline{z}_j \frac{\partial}{\partial \overline{z}_j}
		\bigg),
	\]
	for every $z \in \DIV{n}$.
\end{lemma}
\begin{proof}
	We will follow the notation from Section~\ref{sec:isotypic}. In particular, we will use the expressions given by \eqref{eq:Rvartheta} and \eqref{eq:Atheta}.
	
	As it is well known, we have $\exp(\vartheta R(\pi/2)) = R(\vartheta)$, for every $\vartheta \in \R$. Hence, for every $\theta \in \R^{\ell+1}$ the element $X$ of $\ft_0 \times \so(2)$ given by
	\[
		X = \sum_{j=1}^{\ell+1} \theta_j X_j
	\]
	satisfies for every $r \in \R$
	\[
		\exp(rX) = (A(r\theta'), e^{ir\theta_{\ell+1}}),
	\]
	where $\theta'$ is obtained from $\theta$ by removing the last component. It follows that for $n = 2\ell + 1$ odd, we have
	\[
		\exp(rX) \cdot z = e^{ir\theta_{\ell+1}}
			\begin{pmatrix}
				z_1 \cos(r\theta_1) + z_2 \sin(r\theta_1) \\
				-z_1 \sin(r\theta_1) + z_2 \cos(r\theta_1) \\
				\vdots \\
				z_{2j-1} \cos(r\theta_j) + z_{2j} \sin(r\theta_j) \\
				-z_{2j-1} \sin(r\theta_j) + z_{2j} \cos(r\theta_j) \\
				\vdots \\
				z_{2\ell-1} \cos(r\theta_\ell) + z_{2\ell} \sin(r\theta_\ell) \\
				-z_{2\ell-1} \sin(r\theta_\ell) + z_{2\ell} \cos(r\theta_\ell) \\
				z_{2\ell +1}
			\end{pmatrix},
	\]
	for every $z \in \DIV{n}$. The case of $n = 2\ell$ even is obtained from this expression by removing the last coordinate. Hence, we conclude the result by differentiating the previous expression with respect to $r$ at $0$ and applying \eqref{eq:Xsharpinpartials}.	
\end{proof}

The next result yields a moment map for the action of $\mrT{n} \times \SO(2)$.

\begin{theorem}\label{thm:momentmapTnSO(2)}
	For every $n \geq 3$, and with respect to the identification $\ft_0 \times \so(2) = \R^{\ell+1}$, where $\ell = \lfloor n/2 \rfloor$, a moment map for the action of $\mrT{n} \times \SO(2)$ on $\DIV{n}$ is given by
	\begin{align*}
		\mu = \mu^{\mrT{n} \times \SO(2)} : \DIV{n} &\rightarrow \R^{\ell+1} \\
		\mu(z) &= \frac{1}{\Delta(z)} 
			\sum_{j=1}^\ell i(\overline{z}_{2j-1} z_{2j} 
				- z_{2j-1} \overline{z}_{2j}) e_j 
			+ \frac{|z^\top z|^2 - |z|^2}{\Delta(z)} e_{\ell+1},
	\end{align*}
	for every $z \in \DIV{n}$, where $e_1, \dots, e_{\ell+1}$ is the canonical basis of $\R^{\ell+1}$, and $\Delta(z) = 1 + |z^\top z|^2 - 2|z|^2$.
\end{theorem}
\begin{proof}
	Let us consider the functions $f_j : \DIV{n} \rightarrow \R$, for $j =1, \dots, \ell+1$, given by
	\begin{align*}
		f_j(z) &= \frac{i(\overline{z}_{2j-1} z_{2j} 
			- z_{2j-1} \overline{z}_{2j})}{\Delta(z)}, \quad \text{ for $j =1, \dots, \ell$,} \\
		f_{\ell+1}(z) &= \frac{|z^\top z|^2 - |z|^2}{\Delta(z)}.
	\end{align*}
	By Lemma~\ref{lem:momentmapbasis}, it is enough to show that these functions are $\mrT{n} \times \SO(2)$-invariant and that
	\[
		\dif f_j = \omega(X^\sharp_j, \cdot),
	\]
	for every $j =1 \dots, \ell+1$.
	
	Let us consider the invariance property. We observe that the expressions $\Delta(z)$ and $|z^\top z|^2 - |z|^2$ are both invariant under the group $\mrT{n} \times \SO(2)$. Hence, it is enough to show that the functions defined on $\DIV{n}$ by
	\[
		\widehat{f}_j(z) = \overline{z}_{2j-1} z_{2j} 
		- z_{2j-1} \overline{z}_{2j}
	\]
	are $\mrT{n} \times \SO(2)$-invariant for every $j =1, \dots, \ell$. It is clear that these functions are $\SO(2)$-invariant. Let us consider $A(\theta) \in \mrT{n}$ for any given $\theta \in \R^\ell$. Then, for every $j = 1, \dots, \ell$ we have
	\begin{align*}
		\widehat{f}_j(A(\theta) z) =\; 
		&\overline{(z_{2j-1} \cos \theta_j + z_{2j} \sin \theta_j)} 
		(-z_{2j-1} \sin \theta_j + z_{2j} \cos \theta_j) \\
		&- (z_{2j-1} \cos \theta_j + z_{2j} \sin \theta_j) 
		\overline{(-z_{2j-1} \sin \theta_j + z_{2j} \cos \theta_j)} \\
		=\;& (|z_{2j}|^2 -|z_{2j-1}|^2) \cos \theta_j \sin \theta_j
			+ \overline{z}_{2j-1} z_{2j} \cos^2 \theta_j
			- z_{2j-1} \overline{z}_{2j} \sin^2 \theta_j \\
		&+ (|z_{2j-1}|^2 - |z_{2j}|^2) \cos \theta_j \sin \theta_j
			- z_{2j-1} \overline{z}_{2j} \cos^2 \theta_j
			+ \overline{z}_{2j-1} z_{2j} \sin^2 \theta_j \\
		=\;& \overline{z}_{2j-1} z_{2j} 
				- z_{2j-1} \overline{z}_{2j} = \widehat{f}_j(z),
	\end{align*}
	and so the required invariance follows.
	
	It remains to show that condition (1) from Lemma~\ref{lem:momentmapbasis} is satisfied by the given functions. We start by computing the values of the $1$-forms $\omega(X^\sharp_j,\cdot)$. Without being explicit about the values of the coefficients $g_{jk}$ obtained in Proposition~\ref{prop:SympForm}, we have
	\[
		\omega(X^\sharp_j, \cdot) 
			= i\sum_{k,l=1}^n g_{kl} \dif z_k \wedge \dif \overline{z}_l (X^\sharp_j, \cdot),
	\]
	for every $j = 1, \dots, \ell+1$. After replacing the values of $X^\sharp_j$ from Lemma~\ref{lem:Xsharp} we obtain the following expressions where $j =1, \dots, \ell$
	\begin{align}
		\omega(X^\sharp_j, \cdot) =\;& 
			\sum_{k=1}^n i(\overline{z}_{2j-1} g_{k,2j}(z)
				-\overline{z}_{2j} g_{k,2j-1}(z)) \dif z_k 
					\label{eq:omegaXj} \\
		&+ 	\sum_{k=1}^n i(z_{2j} g_{2j-1,k}(z)
				-z_{2j-1} g_{2j,k}(z)) \dif \overline{z}_k, \notag \\	
		\omega(X^\sharp_{\ell +1}, \cdot) =\;&
			- \sum_{k,l=1}^n (\overline{z}_l g_{kl}(z) \dif z_k
				+ z_l g_{lk}(z) \dif \overline{z}_k).	\label{eq:omegaXl1}
	\end{align}
	
	Next, we compute $\dif f_j$, for every $j =1, \dots, \ell+1$. In this case we note that from the proof of Proposition~\ref{prop:BergmanMetric} we have
	\begin{equation}\label{eq:glkpartialozk}
		g_{lk}(z) = \frac{\partial}{\partial \overline{z}_k}
			\bigg(
				\frac{\overline{z}_l - z_l \overline{(z^\top z)}}{\Delta(z)}
			\bigg).
	\end{equation}
	From this it follows that for every $j =1, \dots, \ell$ and $k =1, \dots, n$ we have
	\begin{align*}
		i(z_{2j} g_{2j-1,k}(z) &- z_{2j-1} g_{2j,k}(z)) = \\
			=\;& i z_{2j} 
				\frac{\partial}{\partial \overline{z}_k}
				\bigg(
				\frac{\overline{z}_{2j-1} - z_{2j-1} \overline{(z^\top z)}}{\Delta(z)}
				\bigg) -
				i z_{2j-1} 
				\frac{\partial}{\partial \overline{z}_k}
				\bigg(
				\frac{\overline{z}_{2j} - z_{2j} \overline{(z^\top z)}}{\Delta(z)}
				\bigg) \\
			=\;&  
				\frac{\partial}{\partial \overline{z}_k}
				\bigg(i z_{2j}
				\frac{\overline{z}_{2j-1} - z_{2j-1} \overline{(z^\top 	z)}}{\Delta(z)}
				-i z_{2j-1}
				\frac{\overline{z}_{2j} - z_{2j} \overline{(z^\top z)}}{\Delta(z)}
				\bigg) \\
			=\;& \frac{\partial}{\partial \overline{z}_k}
				\bigg(
				\frac{i z_{2j}\overline{z}_{2j-1} -i z_{2j-1}\overline{z}_{2j}}{\Delta(z)}
				\bigg) = \frac{\partial f_j}{\partial \overline{z}_k}(z).
	\end{align*}
	Since $f_j$ is real-valued we also have
	\[
		\frac{\partial f_j}{\partial z_k}(z) = 
			i(\overline{z}_{2j-1} g_{k,2j}(z)
				- \overline{z}_{2j} g_{k,2j-1}(z)),
	\]
	where we have used the Hermitian symmetry of the coefficients $g_{kl}$. The computation of these partial derivatives together with \eqref{eq:omegaXj} prove that
	\[
		\dif f_j = \omega(X^\sharp_j,\cdot),
	\]
	for every $j = 1, \dots, \ell$. 
	
	For the remaining case, we note that \eqref{eq:glkpartialozk} yields
	\[
		-z_l g_{lk}(z) = \frac{\partial}{\partial \overline{z}_k}
		\bigg(
		\frac{\overline{z}_l - z_l \overline{(z^\top z)}}{\Delta(z)}
		(-z_l)
		\bigg),
	\]
	for every $k,l=1, \dots, n$. On the other hand, we have
	\[
		\sum_{l=1}^n \frac{\overline{z}_l - z_l \overline{(z^\top z)}}{\Delta(z)}
		(-z_l) 
		= \sum_{l=1}^n \frac{z_l^2 \overline{(z^\top z)} - |z_l|^2 }{\Delta(z)}
		= \frac{|z^\top z|^2  - |z|^2}{\Delta(z)} = f_{\ell+1}(z).
	\]
	These computations show that 
	\[
		\frac{\partial f_{\ell+1}}{\partial \overline{z}_k}(z) =
		-\sum_{l=1}^n z_l g_{lk}(z),
	\]
	and since $f_{\ell+1}$ is real-valued we also have
	\[
		\frac{\partial f_{\ell+1}}{\partial z_k}(z) =
		-\sum_{l=1}^n \overline{z}_l g_{kl}(z).
	\]
	Hence, using \eqref{eq:omegaXl1} we conclude that
	\[
		\dif f_{\ell+1} = \omega(X^\sharp_{\ell+1},\cdot).
	\]
	And this completes the proof.
\end{proof}

We note that the Lie algebra $\so(2)$ naturally identifies with $\R$ through its basic element $X_{\ell+1}$ defined above. Hence, the computations from Theorem~\ref{thm:momentmapTnSO(2)} together with Lemma~\ref{lem:momentmapbasis} yield the next result. The last claim follows from direct inspection.

\begin{corollary}\label{cor:momentmapSO(2)}
	For every $n \geq 3$, and with respect to the identification $\so(2) = \R$, a moment map for the action of $\SO(2)$ on $\DIV{n}$ is given by
	\begin{align*}
		\mu = \mu^{\SO(2)} : \DIV{n} &\rightarrow \R \\
		\mu(z) &= \frac{|z^\top z|^2 - |z|^2}{1 + |z^\top z|^2 - 2|z|^2},
	\end{align*}
	for every $z \in \DIV{n}$. Furthermore, this moment map is $\SO(n) \times \SO(2)$-invariant.
\end{corollary}

\section{Toeplitz operators with moment maps symbols for $\SO(2)$} \label{sec:Toeplitzmomentmap}
In this section we consider special symbols obtained from moment maps. The most general definition that we will use is the following.

\begin{definition}\label{def:momentmapsymbol}
	Let $H$ be an Abelian subgroup of the biholomorphism group of $\DIV{n}$. Assume that $H$ has a moment map $\mu^H : \DIV{n} \rightarrow \fh^*$, where $\fh^*$ is the dual space of the Lie algebra of $H$. Then, a moment map symbol for $H$ or a $\mu^H$-symbol is an essentially bounded function $a \in L^\infty(\DIV{n})$ for which there is some measurable function $f$ so that $a = f\circ \mu^H$. The space of all such $\mu^H$-symbols will be denoted by $L^\infty(\DIV{n})^{\mu^H}$.
\end{definition}

\subsection{Moment map symbols for $\SO(2)$ and commutative $C^*$-algebras}
The main goal of this section, and the work itself, is to construct commutative $C^*$-algebras generated by Toeplitz operators from the subgroup $\SO(2)$ by using moment map symbols instead of invariant symbols. In other words, we use $\mu^{\SO(2)}$-symbols to obtain commutative $C^*$-algebras. That is the content of the next result. One should compare this with Theorem~\ref{thm:comm-invariance} where it is proved that $\SO(2)$-invariant symbols do not yield commuting Toeplitz operators.

\begin{theorem}\label{thm:momentmapSO(2)Toeplitz}
	Let $n \geq 3$ and  $\lambda > n-1$ be arbitrarily given. Then, the $C^*$-algebra $\cT^{(\lambda)}(L^\infty(\DIV{n})^{\mu^{\SO(2)}})$ generated by Toeplitz operators with moment map symbols for $\SO(2)$ is commutative.
	Furthermore, if $a = f \circ \mu^{\SO(2)}$ is a moment map symbol for $\SO(2)$, then,  with the notation from Proposition~\ref{prop:isoSO(n)SO(2)}, the Toeplitz operator $T^{(\lambda)}_a$ preserves the Hilbert direct sum decomposition
	\[
		\cA^2_\lambda(\DIV{n})
		= \bigoplus_{m=0}^{\infty}
		\bigoplus_{\substack{k_1, k_2 \in \N \\ k_1 + 2k_2 = m}} \HH^{k_1}(\C^n) (z^\top z)^{k_2}.
	\]	
	and there exist a sequence $(c_{k_1, k_2}(f,\lambda))_{(k_1, k_2) \in \N^2}$ of complex numbers such that
	\[
		T^{(\lambda)}_a|_{\HH^{k_1}(\C^n) (z^\top z)^{k_2}} = 
		c_{k_1, k_2}(f,\lambda) I_{\HH^{k_1}(\C^n) (z^\top z)^{k_2}},
	\]
	for every $(k_1, k_2) \in \N^2$. In other words, $T^{(\lambda)}_a$ acts by a multiple of the identity on each term.
\end{theorem}
\begin{proof}
	By Corollary~\ref{cor:momentmapSO(2)} the moment map $\mu^{\SO(2)}$ for the $\SO(2)$-action is invariant under the action of the group $\SO(n)\times \SO(2)$. It follows that any given symbol of the form $a = f \circ \mu^{\SO(2)}$ (a moment map symbol for $\SO(2)$) is $\SO(n) \times \SO(2)$-invariant. This implies the inclusion
	\[
		\cT^{(\lambda)}(L^\infty(\DIV{n})^{\mu^{\SO(2)}}) \subset 
		\cT^{(\lambda)}(L^\infty(\DIV{n})^{\SO(n) \times \SO(2)}),
	\]
	for every $\lambda > n-1$, and the commutativity of $\cT^{(\lambda)}(L^\infty(\DIV{n})^{\mu^{\SO(2)}})$ follows from Proposition~\ref{thm:comm-invariance}.
	
	On the other hand, for the symbol $a = f \circ \mu^{\SO(2)}$ as above, Proposition~\ref{prop:HinvHinter} implies that $T^{(\lambda)}_a$ intertwines the action of $\SO(n) \times \SO(2)$ and so it preserves the Hilbert direct sum of $\cA^2_\lambda(\DIV{n})$ in our statement, because it is the corresponding isotypic decomposition obtained in Proposition~\ref{prop:isoSO(n)SO(2)}. Since such decomposition is multiplicity-free, the existence of the sequence follows from Schur's Lemma.
\end{proof}

In the rest of this work we will compute the coefficients from Theorem~\ref{thm:momentmapSO(2)Toeplitz}.

\subsection{The Jordan algebra associated to $\DIV{n}$}
\label{subsec:Jordan}
It is well known that any tube-type bounded symmetric domain has an associated complex Jordan algebra that one can use to define several analytic and geometric properties. We refer to \cite{LoosBSDJ,Upmeier1996} for further details on the general theory. For example, in \cite{DQGesturVolume} it was developed a mechanism to compute the kind of coefficients that appear in Theorem~\ref{thm:momentmapSO(2)Toeplitz}, and the information needed is given in terms of the associated Jordan algebra structure. Hence, in this subsection we will compute some of such information to obtain explicit formulas for the coefficients in Theorem~\ref{thm:momentmapSO(2)Toeplitz}. We will provide most of the required proofs, but we refer to \cite{LoosJP,LoosBSDJ,Upmeier1996} for the well known facts.

For $n \geq 3$, every element $x \in \R^n$ will be decomposed as
\[
	x = x_1 e_1 + x'
\]
where $e_1$ is the first canonical basis vector and $x' = (x_2, \dots, x_n)^\top$. Then, $\R^n$ becomes a Jordan algebra with the product given by
\begin{equation}\label{eq:JordanProdReal}
	x \circ y = (x_1 y_1 + x'\cdot y')e_1 + (x_1 y' + y_1 x').
\end{equation}
It is clear that $e_1$ is a unit of this Jordan algebra.

As usual, we denote $x^2 = x \circ x$, for every $x \in \R^n$. Then, the cone of positive elements in $\R^n$ is defined by
\begin{equation}\label{eq:JordanConeReal}
	\Omega = \{ x^2 \mid x \in \R^n\}^\circ,
\end{equation}
the interior of the subset of squares. This cone yields an order $\prec$ in $\R^n$ given by defining $x \prec y$ if and only if $y-x \in \Omega$. 

\begin{proposition}\label{prop:OmegaCone}
	The cone of positive elements of the Jordan algebra $\R^n$ satisfies
	\[
		\Omega = \{ x \in \R^n \mid x_1 > 0, \; x_1^2 - x'\cdot x'\}.
	\]
\end{proposition}
\begin{proof}
	By taking interiors it is enough to prove that $x$ is the square of some element if and only if $x_1 \geq 0$ and $x_1^2 - x'\cdot x'\geq 0$.
	
	Let us assume that $x = y^2$. Then, we have
	\begin{align*}
		x_1 &= y_1^2 + y'\cdot y' \geq 0 \\
		x_1^2 - x'\cdot x' &= (y_1^2 + y'\cdot y')^2 
			- 4y_1^2 y'\cdot y' \\
			&= (y_1^2 - y'\cdot y')^2 \geq 0.
	\end{align*}

	Conversely, let us consider $x \in \R^n$ such that $x_1 \geq 0$ and $x_1^2 - x'\cdot x' \geq 0$. We need to prove the existence of $y \in \R^n$ such that $y^2 = x$. This is equivalent to solving for $y$ the equations
	\begin{align}
		y_1^2 + y' \cdot y' &= x_1, \label{eq:solveroot1} \\
		2y_1 y' &= x'. \label{eq:solveroot2}
	\end{align}
	If $x_1 = 0$, then we necessarily have $x' = 0$, and a solution to \eqref{eq:solveroot1} and \eqref{eq:solveroot2} is given by $y = 0$. If $x_1 > 0$ and $x' = 0$, then a solution is given by $y = (\sqrt{x}_1, 0, \dots, 0)^\top$. 
	
	Hence, we will assume that $x_1 > 0$ and $x' \not= 0$. Any solution $y$ must satisfy $y_1 \not= 0$. It follows that replacing $y' = \frac{x'}{2y_1}$, the system of equations \eqref{eq:solveroot1} and \eqref{eq:solveroot2} is equivalent to 
	\begin{align*}
		4y_1^4 - 4x_1 y_1^2 + x'\cdot x' &= 0 \\
		2y_1 y' &= x'.
	\end{align*}
	We solve for $y_1$ the first of these equations by choosing
	\[
		y_1 = \sqrt{\frac{x_1 + \sqrt{x_1^2 - x'\cdot x'}}{2}}
	\]
	which is a well defined positive real number by the choice of $x$. Correspondingly, we choose
	\[
		y' = \frac{x'}{2y_1} = \frac{x'}{\sqrt{2(x_1 + \sqrt{x_1^2 - x'\cdot x'})}}.
	\]
	It is now straightforward to check that for such element $y$ we have $y^2 = x$.
\end{proof}

The existence and uniqueness of square roots in $\Omega$ is well known. In the next result we provide an explicit formula for such square roots. Our proof shows the uniqueness for the sake of completeness.

\begin{corollary}\label{cor:JordanRoots}
	In the Jordan algebra $\R^n$, for every $x \in \Omega$ there exists a unique $y \in \Omega$ such that $y^2 = x$. This solution will be denoted $\sqrt{x}$, and it is given by
	\[
		\sqrt{x} =
			\sqrt{\frac{x_1 + \sqrt{x_1^2 - x'\cdot x'}}{2}} e_1 + 
				\frac{x'}{\sqrt{2(x_1 + \sqrt{x_1^2 - x'\cdot x'})}}.
	\]
\end{corollary}
\begin{proof}
	For a given $x \in \Omega$, let us consider the element $y$ given in the statement. As noted in the proof of Proposition~\ref{prop:OmegaCone} we have $y^2 = x$. This choice of $y$ clearly satisfies $y_1 > 0$. We also have
	\begin{align*}
		y_1^2 - y'\cdot y' &= \frac{x_1 + \sqrt{x_1^2 - x'\cdot x'}}{2} - 
			\frac{x'\cdot x'}{2(x_1 + \sqrt{x_1^2 - x'\cdot x'})} \\
			&= \frac{(x_1 + \sqrt{x_1^2 - x'\cdot x'})^2 - x'\cdot x'}{2(x_1 + \sqrt{x_1^2 - x'\cdot x'})} \\
			&= \frac{x_1^2 - x'\cdot x' + x_1\sqrt{x_1^2 - x'\cdot x'}}{x_1 + \sqrt{x_1^2 - x'\cdot x'}} \\
			&= x_1 \sqrt{x_1^2 - x'\cdot x'} > 0.
	\end{align*}
	Hence, it follows from Proposition~\ref{prop:OmegaCone} that $y \in \Omega$, and so the existence has been established.
	
	To establish the uniqueness, we note that it is clear for $x' = 0$. Hence, we will assume that $x' \not= 0$. In this situation, the proof of Proposition~\ref{prop:OmegaCone} shows that any $y \in \R^n$ such that $y^2 = x$ satisfies
	\[
		y_1^2 = \frac{x_1 \pm \sqrt{x_1^2 - x'\cdot x'}}{2}.
	\]
	If we further assume that $y \in \Omega$, we claim that
	\[
		y_1^2 \not= \frac{x_1 - \sqrt{x_1^2 - x'\cdot x'}}{2}.
	\]
	Otherwise we would have
	\[
		y' = \frac{x'}{2y_1} = \pm\frac{x'}{\sqrt{2(x_1 - \sqrt{x_1^2 - x'\cdot x'})}}.
	\]
	Hence, the condition $y_1^2 - y'\cdot y' > 0$ would yield the inequality
	\[
		\frac{x_1 - \sqrt{x_1^2 - x'\cdot x'}}{2} > 
		\frac{x'\cdot x'}{2(x_1 - \sqrt{x_1^2 - x'\cdot x'})},
	\]
	and using that $x_1^2 - x'\cdot x' > 0$, this last inequality is easily seen to be equivalent to
	\[
		\sqrt{x_1^2 - x'\cdot x'} > x_1 > 0
	\]
	which is a contradiction. We conclude that any $y \in \Omega$ such that $y^2 = x$ must satisfy
	\[
		y_1^2 = \frac{x_1 + \sqrt{x_1^2 - x'\cdot x'}}{2},
	\]
	and so we also must have
	\[
		y_1 = \sqrt{\frac{x_1 + \sqrt{x_1^2 - x'\cdot x'}}{2}}.
	\]
	This implies that uniqueness holds as well as the formula in the statement.
\end{proof}

Let us now construct the complex Jordan algebra associated to $\DIV{n}$. This will be basically the complexification of the Jordan algebra structure defined for $\R^n$ above. However, we will require the correct choice of a real form. 

We consider the $\R$-linear map $E : \R^n \rightarrow \C^n$ given by
\begin{equation}\label{eq:Emap}
	E(x) = x_1 e_1 + i x',
\end{equation}
where we use the notation established before. Then, we clearly have the direct sum $\C^n = E(\R^n) \oplus i E(\R^n)$. In other words, $E(\R^n)$ is a real form of $\C^n$. Hence, we can extend the real Jordan algebra structure on $\R^n$ to a complex Jordan algebra structure by complexification. A straightforward computation shows that the corresponding product on $\C^n$ is given by
\begin{equation}\label{eq:JordanProductC}
	z\circ w = (z_1 w_1 - z'\cdot w')e_1 + (z_1 w' + w_1 z').
\end{equation}
As before, for every element $z \in \C^n$, we denote by $z'$ the vector in $\C^{n-1}$ obtained from $z$ by dropping the first coordinate. Also, the real form $E(\R^n)$ defines in $\C^n$ an involution $z \mapsto z^*$ given by the conjugation with respect to such real form. A simple computation shows that we have
\begin{equation}\label{eq:JordanConjC}
	z^* = \overline{z}_1 e_1 - \overline{z}',
\end{equation}
for every $z \in \C^n$.

We will show that $\C^n$ with the product and conjugation given above is precisely the complex Jordan algebra associated to $\DIV{n}$ through the theory that establishes the correspondence between tube-type bounded symmetric domains and complex Jordan algebras. This fact is basically contained in Example~1.5.37 from \cite{Upmeier1996}, but we will sketch its proof below for the sake of completeness. Further details on such correspondence between Jordan structures and bounded symmetric domains can be found in \cite{LoosBSDJ,LoosJP,Upmeier1996}. We will use some facts and results from these references. We now state the main correspondence result that we will use, which follows from pages v-viii and 3.6 from \cite{LoosBSDJ} (see also Subsection~1.9 and Proposition~1.11 from~\cite{LoosJP}). The definitions of the algebraic objects considered in the next results and their proofs can be found in the references just mentioned.

\begin{proposition}\label{prop:AssociatedJordan}
	Let $D$ be a circled bounded symmetric domain in $\C^n$ with (weightless) Bergman kernel $K_D$, and consider the coefficients
	\[
		C_{jkml} = c\eval[3]{\frac{\partial^4 \log K_D(z,z)}{\partial z_j \partial\overline{z}_k \partial z_m \partial \overline{z}_l}}_{z=0}
	\]
	where $j,k,m,l = 1, \dots, n$ and $c > 0$ is some normalizing constant. Assume that $\C^n$ carries a conjugation $v \mapsto v^*$ from a given real form. If we define on $\C^n$ a triple product by the expression
	\[
		\{u v^* w\} = 
			\sum_{j,k,m,l=1}^n C_{jkml}u_j v^*_k w_m e_l,
	\]
	for every $u, v, w \in \C^n$, then $\C^n$ becomes the Jordan pair associated to $D$. Furthermore, if $e \in \C^n$ is a maximal tripotent, then the assignment 
	\[
		(z,w) \mapsto z \circ w = \frac{1}{2} \{z e^* w\}
	\]
	where $z, w \in \C^n$, defines a Jordan algebra structure on $\C^n$ with unit element $e$. This Jordan algebra structure on $\C^n$ is the Jordan algebra associated to $D$.
\end{proposition}

The previous result allows us to identify the Jordan algebra associated to the domain $\DIV{n}$.

\begin{corollary}\label{cor:JordanDIV}
	The complex Jordan algebra associated to $\DIV{n}$ is obtained by endowing $\C^n$ with the product and conjugation given by \eqref{eq:JordanProductC} and \eqref{eq:JordanConjC}.
\end{corollary}
\begin{proof}
	We choose the normalizing constant $c = \frac{1}{2n}$. A lengthy but straightforward computation shows that the coefficients defined in Proposition~\ref{prop:AssociatedJordan} are given by
	\[
		C_{jkml} = 2(\delta_{jk}\delta_{ml} - \delta_{jm}\delta_{kl} + \delta_{km}\delta_{jl}).
	\]
	Note that one can take over from Proposition~\ref{prop:BergmanMetric} and its proof to obtain this expression. Another straightforward computation shows that the triple product given in Proposition~\ref{prop:AssociatedJordan} satisfies 
	\[
		\frac{1}{2} \{z e_1^* w\} = \frac{1}{2} \{z e_1 w\} = 
		(z_1 w_1 - z'\cdot w')e_1 + (z_1 w' + w_1 z').
	\]
	Finally, it is easy to prove that $e_1$ is maximal tripotent (see Example~1.5.69 from~\cite{Upmeier1996}). Hence, the result follows from Proposition~\ref{prop:AssociatedJordan}.
\end{proof}

\subsection{Spectral integral formulas for moment map symbols for $\SO(2)$}
We will now apply one of the main results from \cite{DQGesturVolume} to our setup to obtain explicit expressions for the coefficients from Theorem~\ref{thm:momentmapSO(2)Toeplitz}. The next result is a restatement of Theorem~4.11 from \cite{DQGesturVolume} for the domain $\DIV{n}$ adapted to our current notation.

\begin{theorem}\label{thm:DQBGesturVolume}
	Let $n \geq 3$ and $\lambda > n-1$ be given. For any symbol $a \in L^\infty(\DIV{n})$ which is $\SO(n) \times \SO(2)$-invariant the Toeplitz operator $T^{(\lambda)}_a$ preserves the Hilbert direct sum decomposition from Proposition~\ref{prop:isoSO(n)SO(2)}, and for every $(k_1, k_2) \in \N^2$ there exists a constant $c_{k_1, k_2}(a,\lambda) \in \C$ such that
	\[
		T^{(\lambda)}_a|_{\HH^{k_1}(\C^n) (z^\top z)^{k_2}} = 
		c_{k_1, k_2}(a,\lambda) I_{\HH^{k_1}(\C^n) (z^\top z)^{k_2}}.
	\]
	Furthermore, we have
	\begin{multline*}
		c_{k_1, k_2}(a,\lambda) = 
		\frac{\displaystyle \int\limits_{\Omega \cap (e_1-\Omega)} a(E(\sqrt{x})) p_1(E(x))^{k_1} p_2(E(x))^{k_2} p_2(e_1 - E(x))^{\lambda -n} \dif x}{\displaystyle \int\limits_{\Omega \cap (e_1-\Omega)} p_1(E(x))^{k_1} p_2(E(x))^{k_2} p_2(e_1 - E(x))^{\lambda -n} \dif x} \\
		=
		\frac{\displaystyle \int\limits_{\Omega \cap (e_1-\Omega)} a(E(\sqrt{x})) (x_1 + x_2)^{k_1} (x_1^2 - x' \cdot x')^{k_2} ((1-x_1)^2 - x'\cdot x')^{\lambda -n} \dif x}{\displaystyle \int\limits_{\Omega \cap (e_1-\Omega)} (x_1 + x_2)^{k_1} (x_1^2 - x' \cdot x')^{k_2} ((1-x_1)^2 - x'\cdot x')^{\lambda -n} \dif x},
	\end{multline*}
	for every $(k_1, k_2) \in \N^2$, where $p_1(z) = z_1 - iz_2$ and $p_2(z) = z^\top z$.
\end{theorem}
\begin{proof}
	We mainly need to compare the notation from \cite{DQGesturVolume} with ours with respect to Theorem~4.11 of that reference. This will be done exclusively for the case of the domain $\DIV{n}$.
	
	In the first place, the results from \cite{DQGesturVolume} consider the polynomials that yield the highest weight vectors for the isotypic decomposition in Proposition~\ref{prop:isoSO(n)SO(2)}. In our case, these are precisely $p_1(z)^{k_1} p_2(z)^{k_2}$, where $(k_1, k_2) \in \N^2$, and in \cite{DQGesturVolume} these are parameterized as $p_1(z)^{\alpha_1 - \alpha_2} p_2(z)^{\alpha_2}$, where $\alpha \in \N^2$ satisfies $\alpha_1 \geq \alpha_2$. A standard change of parameter show these to be equivalent.
	
	On the other hand, Theorem~4.11 from \cite{DQGesturVolume} considers the positive cone in the complex Jordan algebra associated to the domain $\DIV{n}$. The latter has been proved to be precisely the one described in Corollary~\ref{cor:JordanDIV}. This complex Jordan algebra and its positive cone are related to the Jordan algebra $\R^n$ and its corresponding cone (given by \eqref{eq:JordanProdReal} and \eqref{eq:JordanConeReal}, respectively) through the map $E$ given by \eqref{eq:Emap}. Hence, the first integral formula in the statement is equivalent to that in Theorem~4.11 from \cite{DQGesturVolume} by the change of variable $E$, which clearly preserves the Lebesgue measure. Note that we have used that $E$ fixes $e_1$.
	
	Hence, the second integral formula follows by substitution.
\end{proof}

\begin{remark}\label{rmk:Evariablechange}
	Let us consider $x \in \Omega$. Then, Corollary~\ref{cor:JordanRoots} implies that
	\[
		E(\sqrt{x}) =
		\sqrt{\frac{x_1 + \sqrt{x_1^2 - x'\cdot x'}}{2}} e_1 + 
		i\frac{x'}{\sqrt{2(x_1 + \sqrt{x_1^2 - x'\cdot x'})}}.
	\]
	From this a simple computation yields
	\begin{align*}
		|E(\sqrt{x})|^2 &=
			\frac{x_1 + \sqrt{x_1^2 - x'\cdot x'}}{2} + 
		\frac{x'\cdot x'}{2(x_1 + \sqrt{x_1^2 - x'\cdot x'})} = x_1, \\
		E(\sqrt{x})^\top E(\sqrt{x}) &= 
			\frac{x_1 + \sqrt{x_1^2 - x'\cdot x'}}{2} - 
		\frac{x'\cdot x'}{2(x_1 + \sqrt{x_1^2 - x'\cdot x'})}
		= \sqrt{x_1^2 - x' \cdot x'}.
	\end{align*}
	It follows that for every $x \in \Omega$ the following sequence of equivalences hold
	\begin{align*}
		E(\sqrt{x}) \in \DIV{n} 
			&\Longleftrightarrow
				|E(\sqrt{x})|^2 < 1, \quad
				2|E(\sqrt{x})|^2 < |E(\sqrt{x})^\top E(\sqrt{x})|^2 + 1 \\
			&\Longleftrightarrow
				0 < x_1 < 1, \quad 
				2x_1 < x_1^2 - x'\cdot x' +  1 \\
			&\Longleftrightarrow
				0 < x_1 < 1, \quad 
				x' \cdot x' < (1 - x_1)^2 \\
			&\Longleftrightarrow
				x \in \Omega \cap (e_1 - \Omega).
	\end{align*}
	This precisely yields the elements $x \in \R^n$ over which the integral from Theorem~\ref{thm:DQBGesturVolume} is taking place. With the order defined by the cone $\Omega$, this condition can be rewritten as $0 \prec x \prec e_1$.
\end{remark}

The next result provides the spectral integral formulas for the Toeplitz operators with moment map symbols for $\SO(2)$.

\begin{theorem}\label{thm:specintToeplitzSO(2)}
	For every $n \geq 3$ and $\lambda > n-1$, let $a = f\circ \mu^{\SO(2)} \in L^\infty(\DIV{n})$ be a moment map symbol for $\SO(2)$. Then, the Toeplitz operator $T^{(\lambda)}_a$ preserves the Hilbert direct sum decomposition from Proposition~\ref{prop:isoSO(n)SO(2)}, and for $(k_1, k_2) \in \N^2$ this operator acts on $\HH^{k_1}(\C^n) (z^\top z)^{k_2}$ by a constant $c_{k_1,k_2}(f,\lambda) \in \C$ given by
	\begin{multline*}
		c_{k_1,k_2}(f,\lambda) = \\
		\frac{\displaystyle \int\limits_{0 \prec x \prec e_1} f\bigg(\frac{x'\cdot x'}{|x|^2 - 1}\bigg) (x_1 + x_2)^{k_1} (x_1^2 - x' \cdot x')^{k_2} ((1-x_1)^2 - x'\cdot x')^{\lambda -n} \dif x}{\displaystyle \int\limits_{0 \prec x \prec e_1} (x_1 + x_2)^{k_1} (x_1^2 - x' \cdot x')^{k_2} ((1-x_1)^2 - x'\cdot x')^{\lambda -n} \dif x}.
	\end{multline*}	
\end{theorem}
\begin{proof}
	We observe that for every $x$ satisfying $0 \prec x \prec e_1$ and from Remark~\ref{rmk:Evariablechange} we have
	\begin{align*}
		\mu^{\SO(2)}(E(\sqrt{x})) &= \frac{|E(\sqrt{x})^\top E(\sqrt{x})|^2 - |E(\sqrt{x})|^2}{1 + |E(\sqrt{x})^\top E(\sqrt{x})|^2 - 2|E(\sqrt{x})|^2} \\
		&= \frac{x_1^2 - x'\cdot x' - x_1^2}{1 + x_1^2 - x'\cdot x' - 2x_1^2} = \frac{x'\cdot x'}{|x|^2 - 1}.
	\end{align*}
	Hence, the result follows from Theorem~\ref{thm:DQBGesturVolume}.
\end{proof}

\begin{remark}\label{rmk:momentmapcoord}
	One should compare the integral formulas obtained for $\SO(n) \times \SO(2)$-invariant symbols and $\mu^{\SO(2)}$-symbols obtained in Theorems~\ref{thm:DQBGesturVolume} and \ref{thm:specintToeplitzSO(2)}, respectively. In view of the expression for $E(\sqrt{x})$ from Remark~\ref{rmk:Evariablechange}, the formula obtained for a $\mu^{\SO(2)}$-symbol is much more simple. This highlights the fact that the coordinates implicit in the use of moment maps and symplectic geometric are much more well adapted to the analysis of Bergman spaces and their Toeplitz operators. 
\end{remark}

\subsection*{Acknowledgement}
This research was partially supported by a Conacyt scolarship, by SNI-Conacyt and by the Conacyt Grants 166891 and 61517.

\end{document}